\documentclass{article}%

\usepackage{amsmath}%
\usepackage{amsfonts}%
\usepackage{amssymb}%
\usepackage{amsthm}%
\usepackage{graphicx}
\usepackage{a4wide}
\usepackage[utf8]{inputenc}
\usepackage[english]{babel}
\usepackage{enumerate}
\usepackage{placeins}

\usepackage{url}

\usepackage{array}
\usepackage{multirow}
\usepackage[usenames,dvipsnames,svgnames,table]{xcolor}

\usepackage{todonotes}

\usepackage[affil-it]{authblk}

\newtheorem{theorem}{Theorem}
\numberwithin{theorem}{section}

\newtheorem{corollary}[theorem]{Corollary}

\newtheorem{lemma}[theorem]{Lemma}

\newtheorem{proposition}[theorem]{Proposition}

\theoremstyle{definition}
\newtheorem{definition}[theorem]{Definition}
\newtheorem{remark}[theorem]{Remark}

\newcommand{\MT}{\operatorname{MT}}
\newcommand{\SL}{\operatorname{SL}}
\newcommand{\GL}{\operatorname{GL}}
\newcommand{\mC}{\overline{\mathbb{Q}_\ell}}

\newcommand{\ff}{finitely generated subfield of $\mathbb{C}$}
\newcommand{\ffn}{finitely generated field of characteristic zero}
\newcommand{\ffs}{finitely generated subfields of $\mathbb{C}$}
\newcommand{\barK}{\overline{K}}
\newcommand{\kSep}{K^{\operatorname{s}}}

\makeatletter
\def\@maketitle{%
  \newpage
  \null
  \vskip 2em%
  \begin{center}%
  \let \footnote \thanks
    {\Large\bfseries \@title \par}%
    \vskip 1.5em%
    {\normalsize
      \lineskip .5em%
      \begin{tabular}[t]{c}%
        \@author
      \end{tabular}\par}%
    \vskip 1em%
    {\normalsize \@date}%
  \end{center}%
  \par
  \vskip 0.7em}
\makeatother

\begin{document}

\title{On the $\ell$-adic Galois representations attached to nonsimple abelian varieties}
\date{}

\author{Davide Lombardo%
  \thanks{\texttt{davide.lombardo@math.u-psud.fr}}}
\affil{Département de Mathématiques d'Orsay\footnote{Univ. Paris-Sud, CNRS, Université Paris-Saclay, 91405 Orsay, France.}}

\maketitle

{\begin{center} \textsc{Abstract} \end{center}}
{\small We study Galois representations attached to nonsimple abelian varieties over finitely generated fields of arbitrary characteristic. We give sufficient conditions for such representations to decompose as a product, and apply them to prove arithmetical analogues of results shown by Moonen and Zarhin in the context of complex abelian varieties (of dimension at most 5).}

\section{Introduction}
Let $K$ be a field finitely generated over its prime subfield, and let $A$ be an abelian variety over $K$. The action of the absolute Galois group of $K$ on the various Tate modules $T_\ell A$ (for $\ell \neq \operatorname{char} K$) gives a (compatible) family of $\ell$-adic representations of the absolute Galois group of $K$, and most of the relevant information is encoded neatly in a certain family of algebraic groups (denoted $H_\ell(A)$ in what follows, cf. definitions \ref{def_Hl} and \ref{def_Hl_p}). It is thus very natural to try and understand the Galois action on nonsimple varieties in terms of the groups $H_\ell$; the main results of this paper are several sufficient criteria for the equality $H_\ell(A \times B) \cong H_\ell(A) \times H_\ell(B)$ to hold. We start by discussing the case $\operatorname{char} K=0$, which is technically simpler, and prove for example the following $\ell$-adic version, and mild generalization, of a Hodge-theoretical result of Hazama \cite{HazNonsimple}:

{
\renewcommand{\thetheorem}{\ref{thm_Hazama}}
\begin{theorem}
Let $K$ be a \ffn, $A_1$ and $A_2$ be $K$-abelian varieties, and $\ell$ be a prime number. For $i=1,2$ let $\mathfrak{h}_{i}$ be the Lie algebra of $H_\ell(A_i)$. Suppose that the following hold:
\begin{enumerate}
\item for $i=1,2$, the algebra $\mathfrak{h}_i$ is semisimple, so that we can write $\mathfrak{h}_{i} \otimes \mC \cong \mathfrak{h}_{i,1} \oplus \cdots \oplus \mathfrak{h}_{i,n_i}$, where every $\mathfrak{h}_{i,j}$ is simple;

\item for $i=1,2$, there exists a decomposition $V_\ell(A_i) \otimes \mC \cong V_{i,1} \oplus \cdots \oplus V_{i,n_i}$ such that the action of $\mathfrak{h}_{i} \otimes \mC \cong \mathfrak{h}_{i,1} \oplus \cdots \oplus \mathfrak{h}_{i,n_i}$ on $V_{i,1} \oplus \cdots \oplus V_{i,n_i}$ is componentwise and $\mathfrak{h}_{i,j}$ acts faithfully on $V_{i,j}$;

\item for all distinct pairs $(i,j)$ and $(i',j')$ for which there exists an isomorphism $\varphi: \mathfrak{h}_{i,j} \to \mathfrak{h}_{i',j'}$ there is an irreducible $\mathfrak{h}_{i,j}$-representation $W$ such that all simple $\mathfrak{h}_{i,j}$-submodules of $V_{i,j}$ and of $\varphi^*\left( V_{i',j'} \right)$ are isomorphic to $W$, and the highest weight defining $W$ is stable under all automorphisms of $\mathfrak{h}_{i,j}$.

\end{enumerate}

Then either $\operatorname{Hom}_{\barK}(A_1,A_2) \neq 0$ or $H_\ell(A_1 \times A_2) \cong H_\ell(A_1) \times H_\ell(A_2)$.
\end{theorem}
\addtocounter{theorem}{-1}
}

From this theorem we deduce many easily applicable criteria, including for example the following result on low-dimensional abelian varieties.
{
\renewcommand{\thetheorem}{\ref{cor_OkWithECAndSurfaces}}
\begin{corollary}
Let $K$ be a \ff\ and $A_1, \ldots, A_n$ be absolutely simple $K$-abelian varieties of dimension at most 2, pairwise non-isogenous over $\barK$. Let $k_1,\ldots,k_n$ be positive integers and $A$ be a $K$-abelian variety that is $\barK$-isogenous to $\prod_{i=1}^n A_i^{k_i}$. Then we have $H_\ell\left( A \right) \cong \prod_{i=1}^n H_\ell(A_i)$, and the Mumford-Tate conjecture holds for $A$.
\end{corollary}
\addtocounter{theorem}{-1}
}

On the other hand, as the conditions in theorem \ref{thm_Hazama} are often not easy to check, it would be desirable to describe families of abelian varieties for which they are known to hold; in this direction we prove a result inspired by a paper of Ichikawa \cite{Ichikawa1991}, where a sufficient criterion is given for the equality $H(A \times B) \cong H(A) \times H(B)$ to hold for the Hodge groups of complex abelian varieties. The criterion is expressed in terms of the \textit{relative dimensions} of the factors:

\begin{definition}\label{def_RelDim}
Let $K$ be any field and $A$ be an absolutely simple $K$-abelian variety, so that $\operatorname{End}_{\barK}^0(A)=\operatorname{End}_{\barK}(A) \otimes_{\mathbb{Z}} \mathbb{Q}$ is a division algebra, with center a number field $E$ (either totally real or CM) of degree $e$ over $\mathbb{Q}$. The degree of $\operatorname{End}_{\barK}^0(A)$ over $E$ is a perfect square, which we write as $d^2$; by \textit{type} of $A$ we mean the type of $\operatorname{End}_{\barK}^0(A)$ in the Albert classification. The relative dimension of $A$ is then given by
\[
\operatorname{rel dim}(A) = \begin{cases} \displaystyle \frac{\dim A}{de}, \mbox{ if } A \mbox{ is of type I, II or III}\\[5pt] \displaystyle \frac{2\dim A}{de}, \mbox{ if } A \mbox{ is of type IV}\end{cases}
\]
Equivalently, the relative dimension of $A$ is given by $\displaystyle \frac{\dim A}{de_0}$, where $e_0=[E_0:\mathbb{Q}]$ is the degree over $\mathbb{Q}$ of the maximal totally real subfield $E_0$ of $E$. 
Note that $d=1$ if $A$ is of type I, and $d=2$ if $A$ is of type II or III.
\end{definition}

A Ribet-style lemma (proved in section \ref{subsect_Lemmas}) that slightly generalizes results found in the literature, combined with techniques due to Pink \cite{Pink} and Larsen-Pink \cite{MR1074479}, allows us to prove the following $\ell$-adic analogue of Ichikawa's theorem, which has exactly the same form as the corresponding Hodge-theoretical result: 

{
\renewcommand{\thetheorem}{\ref{thm_Ichikawa_ell}}
\begin{theorem}
Let $K$ be a \ffn\ and $A'_i, A''_j$ (for $i=1,\ldots,n$ and $j=1,\ldots,m$) be absolutely simple $K$-abelian varieties of odd relative dimension that are pairwise non-isogenous over $\barK$. Suppose every $A'_i$ is of type I, II or III in the sense of Albert, and every $A''_j$ is of type IV. Let $A$ be a $K$-abelian variety that is $\overline{K}$-isogenous to $\prod_{i=1}^n A'_i \times \prod_{j=1}^m A''_j$: then
\[
H_\ell \left( A \right) \cong \prod_{i=1}^n H_\ell\left( A'_i\right) \times H_\ell\left( \prod_{j=1}^m A''_j \right).
\]
\end{theorem}
\addtocounter{theorem}{-1}
}

In section \ref{sect_PosChar} we then discuss to which extent the previous results apply to finitely generated fields of positive characteristic. It turns out that in this setting the most natural definition of $H_\ell(A)$ is different, and that some additional technical hypotheses must be added to our main results. Theorems \ref{thm_Hazama_p} and \ref{thm_Ichikawa_p} are positive-characteristic versions of theorems \ref{thm_Hazama} and \ref{thm_Ichikawa_ell} respectively; they are slightly weaker than their characteristic-zero counterparts, but are still qualitatively very similar.

\smallskip

Finally, in section \ref{sect_MoonenZarhin} we apply our results to nonsimple varieties of dimension at most 5 defined over finitely generated subfields of $\mathbb{C}$; by studying the product structure of $H_\ell$ we prove the Mumford-Tate conjecture for most such varieties, and in all cases we are able to reproduce in the arithmetical setting results obtained in \cite{MZ} for their Hodge group. Note that \cite{MZ} makes ample use of compactness arguments (for real semisimple groups) that are not available in the $\ell$-adic context and thus need to be replaced in our setting.

\smallskip
\noindent \textbf{Acknowledgements.} It is a pleasure to thank my advisor, N. Ratazzi, for encouraging me to look into the matters studied in this paper, for the many valuable discussions and for his careful reading of this document. I also wish to thank the anonymous referee for his or her many detailed comments and extremely valuable suggestions, that led for example to the introduction of section \ref{sect_PosChar}. This work was partially supported by the FMJH through the grant no ANR-10-CAMP-0151-02 in the ``Programme des Investissements d'Avenir".

\section{Preliminaries}
\subsection{Notation}
Throughout the paper the letter $A$ will be reserved for an abelian variety defined over a field $K$, which we suppose to be finitely generated (over its prime subfield). A field $K$ will be said to be a ``finitely generated subfield of $\mathbb{C}$'' if it is finitely generated over $\mathbb{Q}$ and a distinguished embedding $\sigma:K \hookrightarrow \mathbb{C}$ has been fixed. If $A$ is an abelian variety defined over a finitely generated subfield of $\mathbb{C}$, we will write $A_{\mathbb{C}}$ for the base-change of $A$ to $\mathbb{C}$ along $\sigma$; the symbol $V(A)$ will then denote the first homology group $H_1 \left(A_{\mathbb{C}}(\mathbb{C}),\mathbb{Q}\right)$. We will also denote $\ell$ a prime number, and write $V_\ell(A)$ for $T_\ell(A) \otimes \mathbb{Q}_\ell$, where $T_\ell(A)$ is as usual the $\ell$-adic Tate module of $A$. 

\medskip

If $G$ is an algebraic group we shall write $G^{\operatorname{der}}$ for its derived subgroup, $Z(G)$ for the connected component of its center, and $G^0$ for the connected component of the identity; when $\mathfrak{h}$ is a reductive Lie algebra we shall write $\mathfrak{h}^{\operatorname{ss}}$ for its semisimple part. Finally, if $\varphi: \mathfrak{g} \to \mathfrak{h}$ is a morphism of Lie algebras and $\rho: \mathfrak{h} \to \mathfrak{gl}(V)$ is a representation of $\mathfrak{h}$, we denote $\varphi^*(V)$ the representation $\rho \circ \varphi$ of $\mathfrak{g}$.

\smallskip

\begin{definition}\label{def_StandardRep}
When $\mathfrak{h}$ is a classical Lie algebra (i.e. of Lie type $A_l, B_l, C_l$, or $D_l$), we call \textbf{standard representation} of $\mathfrak{h}$ the one coming from the defining representation of the corresponding algebraic group. It is in all cases the representation with highest weight $\varpi_1$ (in the notation of Bourbaki \cite[Planches I-IV]{Bourbaki79}).
\end{definition}

\subsection{The Hodge group}
We now briefly recall the notion of Hodge group of an abelian variety (defined over an arbitrary subfield $F$ of $\mathbb{C}$), referring the reader to \cite{Moonen_noteson} for more details. To stress that $F$ need not be finitely generated, we depart from our standard notation $A$ and denote $X$ an abelian variety defined over $F$; we denote $X_\mathbb{C}$ the base-change of $X$ to $\mathbb{C}$. The $\mathbb{Q}$-vector space $V(X)=H_1 \left(X_{\mathbb{C}}(\mathbb{C}),\mathbb{Q}\right)$ is naturally endowed with a Hodge structure of type $(-1,0)\oplus(0,-1)$, that is, a decomposition of $\mathbb{C}$-vector spaces $V(X) \otimes \mathbb{C} \cong V(X)^{-1,0} \oplus V(X)^{0,-1}$ such that $\overline{V(X)^{-1,0}}=V(X)^{0,-1}$.

Let $\mu_\infty:\mathbb{G}_{m,\mathbb{C}} \to \GL \left(V(X)_\mathbb{C}\right)$ be the unique cocharacter such that $z \in \mathbb{C}^*$ acts as multiplication by $z$ on $V(X)^{-1,0}$ and trivially on $V(X)^{0,-1}$. The \textbf{Mumford-Tate} group of $X$ is the $\mathbb{Q}$-Zariski closure of the image of $\mu_\infty$, that is to say the smallest $\mathbb{Q}$-algebraic subgroup $\MT(X)$ of $\GL(V(X))$ such that $\mu_\infty$ factors through $\MT(X)_\mathbb{C}$. It is not hard to show that $\MT(X)$ contains the torus of homotheties in $\GL(V(X))$.
\begin{definition}
The \textbf{Hodge group} of $X$ is $H(X)=\left(\MT(X) \cap \operatorname{SL}(V(X))\right)^0$.
\end{definition}

\begin{remark}
The group $\MT(X)$ can be recovered from the knowledge of $H(X)$: indeed, $\MT(X)$ is the almost-direct product of $\mathbb{G}_m$ and $H(X)$ inside $\operatorname{GL}(V(X))$, where $\mathbb{G}_m$ is the central torus of homotheties. 
\end{remark}
It is well known that the group $H(X)$ is connected and reductive, and that there is an isomorphism $\operatorname{End}_{\overline{F}}^0\left(X \right) \cong \operatorname{End}(V(X))^{H(X)}$. Moreover, if $\lambda$ is a polarization of $X_\mathbb{C}$ and $\varphi$ is the bilinear form induced on $V(X)$ by $\lambda$, the group $H(X)$ is contained in $\operatorname{Sp}(V(X),\varphi)$. It is also easy to show that when the $F$-abelian varieties $X_1$ and $X_2$ are isogenous over $\mathbb{C}$ the groups $H(X_1)$ and $H(X_2)$ are isomorphic, and that when $X_\mathbb{C}$ has no simple factor of type IV the group $H(X)$ is semisimple. Finally, we also have some information on the behaviour of $H(X)$ with respect to products: 
\begin{proposition}\label{prop_Products} Let $F$ be a subfield of $\mathbb{C}$ and $X_1,X_2$ be abelian varieties defined over $F$. The group $H(X_1 \times X_2)$ is contained in $H(X_1) \times H(X_2)$, and it projects surjectively on both factors. 

Let $X_1, \ldots, X_k$ be absolutely simple $F$-abelian varieties that are pairwise non-isogenous over $\mathbb{C}$, and let $n_1,\ldots,n_k$ be positive integers. The groups $H(X_1^{n_1} \times \cdots \times X_k^{n_k})$ and $H(X_1 \times \cdots \times X_k)$ are isomorphic.
\end{proposition}

\subsection{The groups $H_\ell(A)$}\label{sect_Statement}
Let now $K$ be a \ffn, $A$ be an abelian variety defined over $K$, and $\ell$ be a prime number; recall that we set $V_\ell(A)=T_\ell(A) \otimes \mathbb{Q}_\ell$. The action of $\operatorname{Gal}\left(\overline{K}/K\right)$ on the torsion points of $A$ induces a representation $\rho_\ell:\operatorname{Gal} \left(\barK/K\right) \to \GL(V_\ell(A)) \cong \operatorname{GL}_{2\dim A}(\mathbb{Q}_\ell)$; the Zariski closure of the image of $\rho_\ell$ is called the \textbf{algebraic monodromy group at $\ell$}, and is denoted $G_\ell(A)$. As in the Hodge-theoretical case, it is known that $G_\ell(A)$ contains the homotheties (Bogomolov \cite{Bogomolov}), so that $G_\ell(A)$ is determined by its intersection with $\SL(V_\ell(A))$. This intersection is our main object of study.

\begin{definition}\label{def_Hl}
Let $K$ be a \ffn\ and $A$ be a $K$-abelian variety. We set $H_\ell(A)=\left(G_\ell(A) \cap \SL(V_\ell(A)) \right)^0$.
\end{definition}

Suppose now that we have fixed an embedding $K \hookrightarrow \mathbb{C}$, so that we can speak of the Hodge group of $A$. The Mumford-Tate conjecture predicts that the group $H_\ell(A)$ should be an $\ell$-adic analogue of $H(A)$, and the two groups are indeed known to share many important properties. It is clear by definition that $H_\ell(A)$ is connected; furthermore, by the comparison isomorphism of étale cohomology we can write $V_\ell(A) \cong V(A) \otimes_{\mathbb{Q}} \mathbb{Q}_\ell$, and since $V(A)$ is equipped with a bilinear form $\varphi$ (induced by a polarization) we obtain by extension of scalars a bilinear form $\varphi_\ell$ on $V_\ell(A)$. It is then possible to show that the inclusion $H_\ell(A) \subseteq \operatorname{Sp}(V_\ell(A),\varphi_\ell)$ holds.

Deeper properties of $H_\ell(A)$ are intimately related to Tate's conjecture for abelian varieties, and we summarize them in the following theorem:
\begin{theorem}{(Faltings \cite{MR718935}, \cite{MR766574})}\label{thm_TateConjecture} Let $K$ be a \ffn, $\ell$ be a prime number, and $A,B$ be $K$-abelian varieties. Then $G_\ell(A)$ is a reductive group, and we have
\[
\operatorname{Hom}_{\mathbb{Q}_\ell[G_\ell(A \times B)]} \left( V_\ell(A), V_\ell(B) \right) \cong \operatorname{Hom}_K(A,B) \otimes \mathbb{Q}_\ell.
\]
In particular we have $\operatorname{End}(V_\ell(A))^{G_\ell(A)} \cong \operatorname{End}_K(A) \otimes_{\mathbb{Z}} \mathbb{Q}_\ell$.
\end{theorem}

\begin{corollary}\label{cor_TateConjecture}
Let $K$ be a \ffn, $A$ and $B$ be abelian varieties defined over $K$, $\ell$ be a prime number, and $\mathfrak{h}_\ell$ be the Lie algebra of $H_\ell(A \times B)$. Suppose $\operatorname{Hom}_{\mathfrak{h}_\ell} \left(V_\ell(A),V_\ell(B)\right) \neq 0$: then $\operatorname{Hom}_{\barK}(A,B) \neq 0$.
\end{corollary}
\begin{proof}
There is a finite extension $K'$ of $K$ such that the Zariski closure $G_\ell$ of the image of the representation $\operatorname{Gal}\left(\overline{K'}/K' \right) \to \operatorname{Aut}\left( V_\ell(A \times B) \right)$ is connected. We want to show that $\operatorname{Hom}_{K'}(A,B) \neq 0$. By the previous theorem it is enough to prove that $\operatorname{Hom}_{\mathbb{Q}_\ell[G_\ell]} \left(V_\ell(A),V_\ell(B)\right)$ is nontrivial. As $G_\ell$ is connected, an element of $\operatorname{Hom}\left(V_\ell(A),V_\ell(B)\right)$ is $G_\ell$-equivariant if and only if it is equivariant for the action of the Lie algebra $\mathfrak{g}_\ell$ of $G_\ell$. On the other hand, we know there is an isomorphism $\mathfrak{g}_\ell \cong \mathfrak{h}_\ell \oplus \mathbb{Q}_\ell$, where the factor $\mathbb{Q}_\ell$ corresponds to the homotheties. Since any linear map commutes with the action of the homotheties we have $\operatorname{Hom}_{\mathbb{Q}_\ell[G_\ell]} \left(V_\ell(A_1),V_\ell(A_2)\right) \cong \operatorname{Hom}_{\mathfrak{h}_\ell} \left(V_\ell(A_1),V_\ell(A_2)\right)$, and the latter space is nontrivial by hypothesis. Thus $\operatorname{Hom}_{K'}\left(A_1,A_2\right)$, and a fortiori $\operatorname{Hom}_{\barK}\left(A_1,A_2\right)$, are both nontrivial.% \neq 0$.
\end{proof}

Notice furthermore that the group $H_\ell(A)$ is unchanged by finite extensions of the base field $K$, and that if $A, B$ are $K$-abelian varieties that are $\barK$-isogenous we have $H_\ell(A) \cong H_\ell(B)$.

Moreover, $H_\ell(A)$ is semisimple when $A_{\barK}$ does not have any simple factor of type IV (the proof of this fact being the same as for Hodge groups, cf. again \cite{Moonen_noteson}, especially proposition 1.24), and it has the same behaviour as $H(A)$ with respect to products:

\begin{proposition}\label{prop_ProductsEll}
Let $K$ be a \ffn\ and $A_1,A_2$ be $K$-abelian varieties. The group $H_\ell(A_1 \times A_2)$ is contained in $H_\ell(A_1) \times H_\ell(A_2)$, and it projects surjectively on both factors. 

Let $A_1, \ldots, A_k$ be absolutely simple $K$-abelian varieties that are pairwise non-isogenous over $\barK$, and let $n_1,\ldots,n_k$ be positive integers. The groups $H_\ell(A_1^{n_1} \times \cdots \times A_k^{n_k})$ and $H_\ell(A_1 \times \cdots \times A_k)$ are isomorphic.
\end{proposition}

\medskip

We also have some information about the structure of $V_\ell(A)$ as a representation of $H_\ell(A)$:

\begin{theorem}{(Pink, \cite[Corollary 5.11]{Pink})}\label{thm_Pink} Let $K$ be a \ffn, $A$ be a $K$-abelian variety, $\ell$ be a prime number, and $\mathfrak{h}_\ell(A)$ be the Lie algebra of $H_\ell(A)$. Write $\mathfrak{h}_\ell(A) \otimes \mC \cong \mathfrak{c} \oplus \bigoplus_{i=1}^n \mathfrak{h}_i$, where $\mathfrak{c}$ is abelian and each $\mathfrak{h}_i$ is simple. Let $W$ be a simple submodule of $V_\ell(A) \otimes \mC$ for the action of $(\mathfrak{h}_\ell(A) \otimes \mC)$, decomposed as $W \cong C \otimes \bigotimes_{i=1}^n W_i$, where each $W_i$ is a simple module over $\mathfrak{h}_i$ and $C$ is a 1-dimensional representation of $\mathfrak{c}$. Then:

\begin{enumerate}
\item each $\mathfrak{h}_i$ is of classical type (i.e. of Lie type $A_l, B_l, C_l$ or $D_l$ for some $l$);
\item if $W_i$ is nontrivial, then the highest weight of $\mathfrak{h}_i$ in $W_i$ is minuscule.
\end{enumerate}
\end{theorem}

\begin{remark}
This theorem is stated in \cite{Pink} only for number fields. The version for finitely generated fields follows easily by a specialization argument (cf. also proposition \ref{prop_Specialization} below).
\end{remark}

For the reader's convenience and future reference, we reproduce the full list of minuscule weights for classical Lie algebras, as given for example in \cite{Bourbaki79} (Chapter 8, Section 3 and Tables 1 and 2); the last column of this table contains $+1$ if the corresponding representation is orthogonal, $-1$ if it is symplectic, and 0 if it is not self-dual.

\begin{table}[!ht]
\noindent   \makebox[\textwidth]{ \begin{tabular}{  l  l  l  l }
     \hline
     \textbf{Root system} & \textbf{Minuscule weight} & \textbf{Dimension} & \textbf{Duality properties} \\ \hline \\[-0.3cm]

		     $A_l \, (l \geq 1)$ & $\omega_r, 1 \leq r \leq l$ & $\displaystyle \binom{l+1}{r}$ & \begin{tabular}{@{}l@{}} $(-1)^r$, if $\displaystyle r=\frac{l+1}{2}$\\ 0, if $\displaystyle r \neq \frac{l+1}{2}$ \end{tabular} \\ \\[-0.3cm] \hline

		$B_l \, (l \geq 2)$ & $\omega_l$ & $2^l$ & \begin{tabular}{@{}l@{}} $+1, \mbox{ if } l \equiv 3,0 \pmod 4$ \\ $-1, \mbox{ if } l \equiv 1,2 \pmod 4$ \end{tabular} \\ \hline

     $C_l \, (l \geq 3)$ & $\omega_1$ & $2l$ & $-1$ \\ \hline

\multirow{2}{*}[-0.8cm]		{$D_l \, (l \geq 4)$} & $\omega_1$ & $2l$ & $+1$  \\ \cline{2-4}
                                        & $\omega_{l-1}, \omega_l$ & $2^{l-1}$ & \begin{tabular}{@{}l@{}} $+1, \mbox{ if } l \equiv 0 \pmod 4$ \\ $-1, \mbox{ if } l \equiv 2 \pmod 4$ \\ $0, \mbox{ if } l \equiv 1 \pmod 2$ \end{tabular} \\ \hline
     \end{tabular}
		} 
\caption{Minuscule weights}
\end{table}

\subsection{Known results towards the Mumford-Tate conjecture}\label{sect_Results}
Let $K$ be again a field finitely generated over $\mathbb{Q}$, and $A$ be an abelian variety over $K$. Fix any embedding $\sigma : K \hookrightarrow \mathbb{C}$, so that we can regard $K$ as a subfield of $\mathbb{C}$, and the Mumford-Tate and Hodge groups of $A$ are defined. The celebrated Mumford-Tate conjecture predicts that the equality $G_\ell(A)^0 = \MT(A) \otimes \mathbb{Q}_\ell$ should hold for every prime $\ell$; equivalently, for every $A$ and $\ell$ we should have $H_\ell(A) \cong H(A) \otimes \mathbb{Q}_\ell$. Note that both sides of this equality are invariant under finite extensions of $K$ and isogenies: in particular, if $A$ and $B$ are $K$-abelian varieties that are $\overline{K}$-isogenous, the conjecture holds for $A$ if and only if it holds for $B$.

Even though the general case of the conjecture is still wide open, many partial results have proven, and we shall now recall a number of them that we will need in what follows. Let us start with the following proposition, which allows a reduction of the problem to the case of $K$ being a number field:
\begin{proposition}{(Serre, Noot, \cite[Proposition 1.3]{MR1355123})}\label{prop_Specialization}
Let $\ell$ be a prime, $K$ be a finitely generated subfield of $\mathbb{C}$ and $A$ be a $K$-abelian variety. There exist a number field $L$, a specialization $B$ of $A$ over $L$, and identifications $H_1(A_\mathbb{C}(\mathbb{C}),\mathbb{Q}) \cong H_1(B_\mathbb{C}(\mathbb{C}),\mathbb{Q})$ and $T_\ell(A) \cong T_\ell(B)$ (compatible with the comparison isomorphism in étale cohomology) such that $\MT (A) =\MT(B)$ and $G_\ell(A) = G_\ell(B)$ under the given identifications.
\end{proposition}

This proposition implies in particular that most results which are known for number fields and depend on a single prime $\ell$ automatically propagate to \ffs. This applies to all the theorems we list in this section, some of which were originally stated only for number fields.

\begin{theorem}{(Piatetskii-Shapiro, Borovoi, Deligne \cite[I, Proposition 6.2]{DeligneInclusion})} Let $K$ be a \ff\ and $A$ be a $K$-abelian variety. For every prime $\ell$ we have the inclusion $G_\ell(A)^0 \subseteq \MT(A) \otimes \mathbb{Q}_\ell$.
\end{theorem}

\begin{theorem}{(Pink, \cite[Theorem 4.3]{MR1339927})}\label{thm_Pink1}
Let $K$ be a \ff\ and $A$ be a $K$-abelian variety. Suppose that the equality $\operatorname{rk}(H(A)) = \operatorname{rk}(H_\ell(A))$ holds for one prime $\ell$: then $H_\ell(A) = H(A) \otimes \mathbb{Q}_\ell$ holds for every prime $\ell$. In particular, if the Mumford-Tate conjecture holds for one prime, then it holds for every prime.
\end{theorem}

\begin{theorem}{(Vasiu, \cite[Theorem 1.3.1]{Vasiu}; cf. also Ullmo-Yafaev, \cite[Corollary 2.11]{Ullmo})}\label{thm_CenterMT}
Let $K$ be a \ff\ and $A$ be a $K$-abelian variety. For every prime $\ell$ we have $Z(H_\ell(A)) \cong Z(H(A)) \otimes \mathbb{Q}_\ell$. In particular, the Mumford-Tate conjecture is true for CM abelian varieties.
\end{theorem}
\begin{remark}
The CM case of the Mumford-Tate conjecture was first proved by Pohlmann \cite{MR0228500}.
\end{remark}

The following proposition follows immediately upon combining the previous three theorems:
\begin{proposition}\label{prop_EnoughToConsiderSS}
Let $K$ be a \ff\ and $A$ be a $K$-abelian variety. Suppose that for one prime number $\ell$ we have $\operatorname{rk}(H(A)^{\operatorname{der}}) \leq \operatorname{rk}(H_\ell(A)^{\operatorname{der}})$: then the Mumford-Tate conjecture holds for $A$. The same is true if (for some prime $\ell$) we have $\operatorname{rk} H(A) \leq \operatorname{rk} H_\ell(A)$.
\end{proposition}

\medskip

In a different direction, many results are known for absolutely simple abelian varieties of specific dimensions:

\begin{theorem}{(Serre, \cite{SerreAbelianRepr})}\label{thm_EC}
The Mumford-Tate conjecture is true for elliptic curves (over \ffs).
\end{theorem}

\begin{theorem}{(Tanke'ev, Ribet, \cite[Theorems 1, 2 and 3]{Ribet83classeson})}\label{thm_PrimeDimension}
The Mumford-Tate conjecture is true for absolutely simple abelian varieties of prime dimension (over \ffs).
\end{theorem}

\begin{theorem}{(Moonen, Zarhin, \cite{Moonen95hodgeand})}\label{thm_MT4}
Let $K$ be a \ff\ and $A$ be an absolutely simple $K$-abelian variety of dimension 4. If $\operatorname{End}_{\barK}(A) \neq \mathbb{Z}$, then the Mumford-Tate conjecture holds for $A$. If $\operatorname{End}_{\barK}(A)=\mathbb{Z}$, then either for all primes $\ell$ we have $H_\ell(A) \cong \operatorname{Sp}_{8,\mathbb{Q}_\ell}$ and Mumford-Tate holds for $A$, or else for all $\ell$ the group $H_\ell(A)$ is isogenous to a $\mathbb{Q}_\ell$-form of $\operatorname{SL}_2^3$.
\end{theorem}

\begin{remark} The preprint \cite{MT4} announces a proof of the Mumford-Tate conjecture for absolutely simple abelian fourfolds $A$ with $\operatorname{End}_{\barK}(A)=\mathbb{Z}$. In what follows we shall not need this fact, whose only effect would be to slightly simplify the statement of theorem \ref{thm_MT5}.
\end{remark}

There are some common elements to the proofs of all the dimension-specific results we just listed, and we shall try to capture them in definition \ref{def_GeneralLefschetz} below. We now try to motivate this definition. As the group $H_\ell(A)$ is reductive and connected, most of its structure is encoded by the $\mathbb{Q}_\ell$-Lie algebra $\mathfrak{h}_\ell(A)=\operatorname{Lie}(H_\ell(A))$; extending scalars to $\mC$, this Lie algebra can be written as $\mathfrak{h}_\ell(A) \otimes \mC \cong \mathfrak{c} \oplus \bigoplus_{i=1}^n \mathfrak{h}_i$, with $\mathfrak{c}$ abelian and each $\mathfrak{h}_i$ simple. The proofs of theorems \ref{thm_EC} and \ref{thm_PrimeDimension} yield information about the structure of this Lie algebra:
%The arguments used to prove theorems \ref{thm_EC} and \ref{thm_PrimeDimension} also show:
\begin{proposition}\label{prop_LefschetzType}
Let $K$ be a \ff\ and $A/K$ be an absolutely simple abelian variety whose dimension is either 1 or a prime number. Fix a prime $\ell$ and let $\mathfrak{h}_\ell(A)$ be the Lie algebra of $H_\ell(A)$. Suppose $A$ is not of type IV. Then the following hold:
\begin{itemize}
\item the Lie algebra $\mathfrak{h}_\ell(A) \otimes \mC$ admits a decomposition $\mathfrak{h}_1 \oplus \cdots \oplus \mathfrak{h}_n$, where each simple factor $\mathfrak{h}_i$ is of Lie type $\mathfrak{sp}_k$ for some $k$;

\item
for each $i=1,\ldots,n$ there exists a (not necessarily simple) $\mathfrak{h_i}$-module $W_i$ such that $V_\ell(A) \otimes \mC$ is isomorphic to $W_1 \oplus \cdots \oplus W_n$, the action of $\mathfrak{h}_1 \oplus \ldots \oplus \mathfrak{h}_n$ on $W_1 \oplus \cdots \oplus W_n$ is componentwise, and $\mathfrak{h}_i$ acts faithfully on $W_i$;
\item every module $W_i$ is a direct sum of copies of the standard representation of $\mathfrak{h}_i$ (cf. definition \ref{def_StandardRep}).
\end{itemize}
\end{proposition}

Trying to isolate the essential features of this proposition, and taking into account theorem \ref{thm_Pink}, we are led to the following definition:  %We can then introduce the following

\begin{definition}\label{def_GeneralLefschetz}
Let $K$ be a finitely generated field of characteristic zero, $A/K$ be an abelian variety, and $\mathfrak{h}_\ell(A)$ be the Lie algebra of $H_\ell(A)$. We can write $\mathfrak{h}_\ell(A) \otimes \mC \cong \mathfrak{c} \oplus \mathfrak{h}_1 \oplus \cdots \oplus \mathfrak{h}_n$, where $\mathfrak{c}$ is abelian and each factor $\mathfrak{h}_i$ is simple and (by theorem \ref{thm_Pink}) of classical type. We say that $A$ is of \textbf{general Lefschetz type} (with respect to the prime $\ell$) if it is absolutely simple, not of type IV, and following hold:

\begin{enumerate}
\item
for each $i=1,\ldots,n$ there exists a (not necessarily simple) $\mathfrak{h_i}$-module $W_i$ such that $V_\ell(A) \otimes \mC$ is isomorphic to $W_1 \oplus \cdots \oplus W_n$, where the action of $\mathfrak{h}_1 \oplus \ldots \oplus \mathfrak{h}_n$ on $W_1 \oplus \cdots \oplus W_n$ is componentwise, and $\mathfrak{h}_i$ acts faithfully on $W_i$;
\item if the simple Lie algebra $\mathfrak{h}_i$ is of Lie type $A_l$, the rank $l$ is odd and $W_i$ is a direct sum of copies of $\bigwedge^{\frac{l+1}{2}} \operatorname{Std}$, where $\operatorname{Std}$ is the standard representation of $\mathfrak{h}_i$ (cf. definition \ref{def_StandardRep});
\item if the simple algebra $\mathfrak{h}_i$ is of Lie type $B_l$, the module $W_i$ is a direct sum of copies of the (spinor) representation defined by the highest weight $\omega_l$ (in the notation of \cite[Planches I-IV]{Bourbaki79});
\item if the simple algebra $\mathfrak{h}_i$ is of Lie type $C_l$ or $D_l$, the module $W_i$ is a direct sum of copies of the standard representation of $\mathfrak{h}_i$.
\end{enumerate}
We shall simply say that $A$ is of general Lefschetz type (without further specification) when properties (1)-(4) hold with respect to every prime $\ell$.
\end{definition}

\begin{remark}
As proved in \cite[Lemma 2.3]{Murty}, when $A$ is a complex abelian variety of type I or II the action of the Lefschetz group of $A$ on $V(A) \otimes \mathbb{C}$ has precisely this structure.% (at least when no simple factor is of type $A_l$).
\end{remark}

%\medskip

Several instances of this situation have been studied, for example in a series of papers by Banaszak, Gajda and Kraso\'n. Among various other results, for abelian varieties of type I and II they prove:% the following classification results (recall that $G_\ell(A)=H_\ell(A) \cdot \mathbb{G}_m$, cf. section \ref{sect_Statement}):

\begin{theorem}{(Theorems 6.9 and 7.12 of \cite{BGK12})}\label{thm_BGK12}
Let $K$ be a \ff\ and $A/K$ be an absolutely simple abelian variety of type I or II. Suppose that $h=\operatorname{rel dim}(A)$ is odd: then for every prime $\ell$ the simple factors of $H_{\ell}(A) \otimes \mC$ are of type $\operatorname{Sp}_{2h}$. Furthermore, the Mumford-Tate conjecture holds for $A$.
\end{theorem}

\begin{remark}\label{rmk_TypesIandII}
It is clear from the proof of \cite[Lemma 4.13]{BGK12} that any abelian variety as in theorem \ref{thm_BGK12} is of general Lefschetz type. Moreover, the result also holds for $h=2$: this is not stated explicitly in \cite{BGK12}, but follows essentially from the same proof (cf. also \cite[Theorem 8.5]{MR1156568}, which covers the case of abelian fourfolds of relative dimension 2).
\end{remark}

Another paper by the same authors, \cite{BGK3}, deals with varieties of type III: %, and contains the following statement:

\begin{proposition}\label{prop_TypeIII}
Let $K$ be a \ff\ and $A/K$ be an absolutely simple abelian variety of type III. Suppose that $h=\operatorname{rel dim}(A)$ is odd: then for every $\ell$ the simple factors of $\left(\operatorname{Lie} H_{\ell}(A) \right) \otimes \mC$ are either of type $\mathfrak{so}_{2h}$ or of type $\mathfrak{sl}_{l+1}$, where $l+1$ is a power of 2. Furthermore, $A$ is of general Lefschetz type.
\end{proposition}

\begin{remark} Note that the authors of \cite{BGK3} claim a stronger statement, namely the fact that the simple factors of $H_{\ell}(A) \otimes \mC$ can only be of type $\operatorname{SO}_{2h}$ and that, under the same hypotheses, Mumford-Tate holds for $A$. The proof of \cite[Lemma 4.13]{BGK3}, however, fails to take into account the minuscule orthogonal representations whose dimension is congruent to 2 modulo 4 (those corresponding to algebras of type $\operatorname{sl}_{l+1}$ acting on $\Lambda^{\frac{l+1}{2}} \operatorname{Std}$, when $l\geq 3$ and $l+1$ is a power of 2); as a result, the statements of \cite[Theorems 4.19 and 5.11]{BGK3} need to be amended as we did in proposition \ref{prop_TypeIII}.
\end{remark}

\section{Preliminary lemmas}\label{subsect_Lemmas}
We now start proving some lemmas on algebraic groups and Lie algebras we will repeatedly need throughout the paper. 

\begin{lemma}\label{lemma_InclusionReductive}
Let $G \hookrightarrow G_1 \times G_2$ be an inclusion of algebraic groups over a field of characteristic zero. Suppose that $G,G_1$ and $G_2$ are reductive and connected, and that the projections of $G$ on $G_1$ and $G_2$ are surjective. If $\operatorname{rk} G$ equals $\operatorname{rk}(G_1)+\operatorname{rk}(G_2)$, then the inclusion is an isomorphism.
\end{lemma}
\begin{proof}
We show that $G$ is open and closed in $G_1 \times G_2$. It is closed because every algebraic subgroup is, and it is open since $G$ and $G_1 \times G_2$ have the same Lie algebra by \cite[Lemma 3.1]{HazPowers}.
\end{proof}

\begin{lemma}\label{lemma_CanChooselSoGroupIsSimple}
Let $G$ be a $\mathbb{Q}$-simple algebraic group. If $G$ is semisimple and the number of simple factors of $G_{\overline{\mathbb{Q}}}$ is at most 3, then there is a set of primes $L$ of positive density such that for every $\ell$ in $L$ the group $G_{\mathbb{Q}_\ell}$ is simple.
\end{lemma}
\begin{proof}
Let $n$ be the number of simple factors of $G_{\overline{\mathbb{Q}}}$; if $n=1$ there is nothing to prove, so we can assume $n$ is $2$ or $3$. The permutation action of $\operatorname{Gal}\left(\overline{\mathbb{Q}}/\mathbb{Q} \right)$ on the simple factors of $G_{\overline{\mathbb{Q}}}$ determines a map $\rho: \operatorname{Gal}\left(\overline{\mathbb{Q}}/\mathbb{Q} \right) \to S_n$, and the assumption that $G$ is $\mathbb{Q}$-simple implies that the image of $\rho$ is a transitive subgroup of $S_n$. As $n \leq 3$, we see that the image of $\rho$ contains an $n$-cycle $g$. By the Chebotarev density theorem there exists a set of primes $L$ of positive density such that $\rho \left(\operatorname{Gal}\left( \mC/\mathbb{Q}_\ell \right) \right)$ contains $g$; in particular, for any such $\ell$ the group $\operatorname{Gal}\left( \mC/\mathbb{Q}_\ell \right)$ acts transitively on the simple factors of $G_{\overline{\mathbb{Q}_\ell}}$, so $G_{\mathbb{Q}_\ell}$ is $\mathbb{Q}_\ell$-simple.
\end{proof}

\begin{lemma}\label{lemma_TrueForHodgeImpliesTrueForHl}
Let $K$ be a \ff\ and $A, B$ be $K$-abelian varieties. Suppose $B$ is CM and $H(A \times B) \cong H(A) \times H(B)$. Then we have $H_\ell(A \times B) \cong H_\ell(A) \times H_\ell(B)$ for every prime $\ell$.
\end{lemma}
\begin{proof}
Using the hypothesis and applying theorem \ref{thm_CenterMT} twice we find 

\[
\begin{aligned}
\operatorname{rk} Z\left(H_\ell \left( A \times B \right)\right) & = \operatorname{rk} Z\left(H \left( A \times B \right)\right) \\
& = \operatorname{rk} Z\left(H \left( A \right)\right)+ \operatorname{rk} Z\left(H\left( B \right)\right) \\
& =\operatorname{rk} Z\left(H_\ell \left( A \right)\right)+ \operatorname{rk} Z\left(H_\ell\left( B \right)\right).
\end{aligned}
\]

Furthermore, as $H_\ell(B)$ is a torus, the canonical projection $H_\ell(A \times B) \to H_\ell(A)$ induces an isogeny $H_\ell(A \times B)^{\operatorname{der}} \cong H_\ell(A)^{\operatorname{der}}$, hence $\operatorname{rk} H_\ell(A \times B)^{\operatorname{der}} =\operatorname{rk} H_\ell(A)^{\operatorname{der}}$. Putting these facts together we get $\operatorname{rk} H_\ell(A \times B) = \operatorname{rk} H_\ell(A) + \operatorname{rk} H_\ell(B)$, so the inclusion $H_\ell(A \times B) \hookrightarrow H_\ell(A) \times H_\ell(B)$ is an isomorphism by lemma \ref{lemma_InclusionReductive}.\end{proof}

The next lemma is certainly well-known to experts (a somewhat similar statement is for example \cite[Théorème 7]{MR0387283}, which deals with the case of elliptic curves), but for lack of an accessible reference we include a short proof:

\begin{lemma}\label{lem_SemisimpleTimesCM} Let $K$ be a \ff\ and $A, B$ be $K$-abelian varieties. Suppose $B$ is of CM type and $A_{\overline{K}}$ has no simple factor of type IV. Then we have $H(A \times B) \cong H(A) \times H(B)$, and for every prime $\ell$ we also have $H_\ell(A \times B) \cong H_\ell(A) \times H_\ell(B)$.
\end{lemma}

\begin{proof}
The same proof works for both the Hodge group and the group $H_\ell$, so let us only treat the former. The canonical projections $H(A \times B) \to H(A)$ and $H(A \times B) \to H(B)$ induce isogenies $H(A \times B)^{\operatorname{der}} \cong H(A)^{\operatorname{der}}$ and $Z(H(A \times B)) \cong Z(H(B))$, so we have
\[
\operatorname{rk} H(A \times B)=\operatorname{rk} H(A \times B)^{\operatorname{der}}+\operatorname{rk} Z(H(A \times B)) = \operatorname{rk} H(A)^{\operatorname{der}}+\operatorname{rk} Z(H(B)) = \operatorname{rk} H(A)+\operatorname{rk} H(B)
\]
and we conclude by lemma \ref{lemma_InclusionReductive}.
\end{proof}

\begin{lemma}\label{lemma_CMImpliesMT}
Let $K$ be a \ff\ and $A, B$ be $K$-abelian varieties. Suppose that Mumford-Tate holds for $A$, and that $B$ is CM. Then Mumford-Tate holds for $A \times B$.
\end{lemma}
\begin{proof}
Let $\ell$ be a prime number. As in the previous lemma we have $\operatorname{rk} H_\ell(A \times B)^{\operatorname{der}} = \operatorname{rk} H_\ell(A)^{\operatorname{der}}$ and $\operatorname{rk} H(A \times B)^{\operatorname{der}} = \operatorname{rk} H(A)^{\operatorname{der}}$. Since the Mumford-Tate conjecture holds for $A$, we deduce $\operatorname{rk} H_\ell(A \times B)^{\operatorname{der}} = \operatorname{rk} H_\ell(A)^{\operatorname{der}} = \operatorname{rk} H(A)^{\operatorname{der}} = \operatorname{rk} H(A \times B)^{\operatorname{der}}$, 
and the lemma follows from proposition \ref{prop_EnoughToConsiderSS}.
\end{proof}

\begin{lemma}\label{lemma_ProductImpliesMT}
Let $K$ be a \ff\ and $A_1,\ldots,A_n$ be $K$-abelian varieties. Suppose that Mumford-Tate holds for every $A_i$, and that the equality $H_\ell\left( \prod_{i=1}^n A_i \right) = \prod_{i=1}^n H_\ell(A_i)$ holds for a given prime $\ell$. Then the Mumford-Tate conjecture holds for $\prod_{i=1}^n A_i$.
\end{lemma}
\begin{proof}
The hypothesis implies
\[
\operatorname{rk} H_\ell \left(\prod_{i=1}^n A_i \right) = \sum_{i=1}^n \operatorname{rk} H_\ell(A_i) = \sum_{i=1}^n \operatorname{rk} H(A_i) \geq \operatorname{rk} H\left( \prod_{i=1}^n A_i \right),
\]
and the lemma follows from proposition \ref{prop_EnoughToConsiderSS}.
\end{proof}

\medskip

One of the most important ingredients in our proofs is the following lemma, part of which is originally due to Ribet. The statement we give here is close in spirit to \cite[Lemma 2.14]{Moonen95hodgeand}, but our version is even more general.% Even though part of this lemma is by now classical, we include a complete proof for the convenience of the reader.

\begin{lemma}\label{lemma_RibetLemma}
Let $\mathbf{C}$ be an algebraically closed field of characteristic zero and $V_1, \ldots, V_n$ be finite-dimensional $\mathbf{C}$-vector spaces. Let $\mathfrak{gl}(V_i)$ be the Lie algebra of endomorphisms of $V_i$ and let $\mathfrak{g}$ be a Lie subalgebra of $\mathfrak{gl}(V_1) \oplus \cdots \oplus \mathfrak{gl}(V_n)$. For each $i=1, \cdots, n$ let $\pi_i: \bigoplus_{j=1}^n \mathfrak{gl}(V_j) \to \mathfrak{gl}(V_i)$ be the $i$-th projection and let $\mathfrak{g}_i=\pi_i(\mathfrak{g})$. Suppose that each $\mathfrak{g}_i$ is a simple Lie algebra and that one of the following conditions holds:

\begin{itemize}
\item[(a)]
For every pair of distinct indices $i,j$ the projection $\pi_i \oplus \pi_j : \mathfrak{g} \to \mathfrak{g}_i \oplus \mathfrak{g}_j$ is onto.

\item[(b)]
%For every (abstract) simple Lie algebra $\mathfrak{l}$ let $I(\mathfrak{l})= \left\{ i \in \left\{1, \cdots,n \right\}  \bigm\vert \mathfrak{l} \cong \mathfrak{g}_i \right\}$. For every $\mathfrak{l}$ such that $|I(\mathfrak{l})|>1$ the following conditions are met:
For all indices $i \neq j$ for which there is an isomorphism $\varphi : \mathfrak{g}_i \to \mathfrak{g}_j$ we have the following:

\begin{enumerate}
\item there is an irreducible $\mathfrak{g}_i$-representation $W$ such that all simple $\mathfrak{g}_i$-submodules of $V_i$ and of $\varphi^*\left( V_{j} \right)$ are isomorphic to $W$, and the highest weight defining $W$ is stable under all automorphisms of $\mathfrak{g}_i$;

\item let $I=\left\{ k \in \left\{1,\ldots,n \right\} \bigm\vert \mathfrak{g}_k \cong \mathfrak{g}_i \right\}$; the equality $\operatorname{End}_\mathfrak{g}\left( \bigoplus_{k \in I} V_k \right) \cong \prod_{k \in I} \operatorname{End}_{\mathfrak{g}_k} V_k$ holds.
\end{enumerate}

\end{itemize}

Then $\displaystyle \mathfrak{g}=\bigoplus_{j=1}^n \mathfrak{g}_j$.
\end{lemma}

\begin{remark}
As inner automorphisms preserve every highest weight, in condition (b1) one only needs to check the action of the outer automorphisms (which are finite in number, up to inner automorphisms, since they correspond to automorphisms of the Dynkin diagram). In particular, our conditions (b) generalize those given in \cite[Lemma 2.14]{Moonen95hodgeand}.
\end{remark}

\begin{proof}
The fact that (a) implies the desired equality is classical, cf. the Lemma on pages 790-791 of \cite{MR0457455}. Thus it suffices to show that (b) implies (a). Let us fix a pair $(i,j)$ and consider the projection $\pi_i \oplus \pi_j : \mathfrak{g} \to \mathfrak{g}_i \oplus \mathfrak{g}_j$. Let $\mathfrak{h}$ be the image of this projection and $\mathfrak{k}$ be $\ker\left(\mathfrak{h} \to \mathfrak{g}_i \right)$. Since $\mathfrak{k}$ can be identified to an ideal of $\mathfrak{g}_j$ (which is simple), we either have $\mathfrak{k} \cong \mathfrak{g}_j$, in which case $\mathfrak{h} \cong \mathfrak{g}_i \oplus \mathfrak{g}_j$ as required, or $\mathfrak{k} = \left\{0\right\}$, in which case $\mathfrak{h}$ is the graph of an isomorphism $\mathfrak{g}_i \cong \mathfrak{g}_j$; it is this latter possibility that we need to exclude. If $\mathfrak{g}_i$ and $\mathfrak{g}_j$ are not isomorphic there is nothing to prove, so let us assume $\mathfrak{g}_i \cong \mathfrak{g}_j$, and suppose by contradiction that $\mathfrak{h}$ is the graph of an isomorphism $\varphi: \mathfrak{g}_i \to \mathfrak{g}_j$. Let $\rho_i : \mathfrak{g}_i \to \mathfrak{gl}(V_i)$ and $\rho_j : \mathfrak{g}_j \to \mathfrak{gl}(V_j)$ be the tautological representations of $\mathfrak{g}_i, \mathfrak{g}_j$. By assumption (b1), the simple $\mathfrak{g}_i$-subrepresentations of $\rho_i$ and $\rho_j \circ \varphi$ are isomorphic, so there exists a nonzero morphism of $\mathfrak{g}_i$-representations $\chi_{ij}:V_i \to V_j$. Equivalently, $\chi_{ij}$ is $\mathfrak{h}$-equivariant (recall that $\mathfrak{h}$ is the graph of $\varphi$). Setting $I=\left\{k\in \left\{1,\ldots,n \right\}\bigm\vert\mathfrak{g}_k\cong\mathfrak{g}_i\right\}$, the map
\[
\begin{matrix}
\Psi: & \displaystyle \bigoplus_{k \in I } V_k & \to & \displaystyle \bigoplus_{k \in I} V_k \\
      & (v_{i_1}, \cdots, \underbrace{v_i}_{\mbox{factor } V_i}, \cdots, v_{i_{|I|}}) & \mapsto & ( 0, \cdots, \underbrace{\chi_{ij}(v_i)}_{\mbox{factor } V_j}, \cdots, 0)
      \end{matrix}
\]
then belongs to $\operatorname{End}_{\mathfrak{g}} \left( \bigoplus_{k \in I} V_k\right)$, but does not send every factor to itself, so it is not an element of $\prod_{k \in I} \operatorname{End}_{\mathfrak{g}_k} \left(V_k\right)$. This contradicts condition (b2), so $\mathfrak{g} \to \mathfrak{g}_i \oplus \mathfrak{g}_j$ must be onto, and therefore (b) implies (a) as required.
\end{proof}

\begin{proposition}\label{prop_DifferentSimpleFactors}
Let $K$ be a \ffn, $A, B$ be $K$-abelian varieties and $\ell$ be a prime number. Suppose $H_\ell(A)$ is semisimple and no simple factor of the (semisimple) Lie algebra $\operatorname{Lie}(H_\ell(A)) \otimes \overline{\mathbb{Q}_\ell}$ is isomorphic to a simple factor of $\operatorname{Lie}(H_\ell(B))^{\operatorname{ss}} \otimes \mC$: then $\operatorname{H}_\ell(A \times B) \cong H_\ell(A) \times H_\ell(B)$.
\end{proposition}
\begin{proof}
As the group $H_\ell(A)$ is semisimple, the projection $H_\ell(A \times B) \twoheadrightarrow H_\ell(B)$ induces an isogeny $Z(H_\ell(A \times B)) \cong Z(H_\ell(B))$, so the centers of $H_\ell(A \times B)$ and of $H_\ell(A) \times H_\ell(B)$ have the same rank.
Next consider the semisimple ranks. Let $\mathfrak{h}$, $\mathfrak{h}_A$ and $\mathfrak{h}_B$ be the Lie algebras $\operatorname{Lie}(H_\ell(A \times B))^{\operatorname{ss}} \otimes \overline{\mathbb{Q}_\ell}$, $\operatorname{Lie}(H_\ell(A)) \otimes \overline{\mathbb{Q}_\ell}$ and $\operatorname{Lie}(H_\ell(B))^{\operatorname{ss}} \otimes \overline{\mathbb{Q}_\ell}$ respectively.

Write $\mathfrak{h}_{A} \cong \mathfrak{g}_1 \oplus \cdots \oplus \mathfrak{g}_{n}$ and $\mathfrak{h}_{B}\cong \mathfrak{g}_{n+1} \oplus \cdots \oplus \mathfrak{g}_{n+m}$, with every $\mathfrak{g}_i$ simple. We can consider $\mathfrak{h}$ as a subalgebra of $\bigoplus_{i=1}^n \mathfrak{g}_i \oplus \bigoplus_{j=1}^m \mathfrak{g}_{n+j}$ that projects surjectively onto $\bigoplus_{i=1}^n \mathfrak{g}_i$ and $\bigoplus_{j=1}^m \mathfrak{g}_{n+j}$. In particular, $\mathfrak{h}$ projects surjectively onto each simple factor $\mathfrak{g}_i$. 

Let us show that all the double projections $\mathfrak{h} \to \mathfrak{g}_i \oplus \mathfrak{g}_j$ are onto. If $i,j$ are both at most $n$ (or $i,j$ are both at least $n+1$) this is trivial, so we can assume $i \leq n < j$. But then by assumption $\mathfrak{g}_i$ and $\mathfrak{g}_j$ are nonisomorphic, so by the same argument as in the proof of lemma \ref{lemma_RibetLemma} the projection must be surjective. Lemma \ref{lemma_RibetLemma} now gives $\mathfrak{h} \cong \mathfrak{h}_{A} \oplus \mathfrak{h}_B$, thus implying $\operatorname{rk} \mathfrak{h} = \operatorname{rk} \mathfrak{h}_A + \operatorname{rk} \mathfrak{h}_B$. In terms of groups this leads to
\[
\begin{aligned}
\operatorname{rk} H_\ell(A \times B)  & = \operatorname{rk} H_\ell(A \times B)^{\operatorname{der}} + \operatorname{rk} Z(H_\ell(A \times B)) \\ 
                                        & = \operatorname{rk} H_\ell(A)^{\operatorname{der}} + \operatorname{rk} H_\ell(B)^{\operatorname{der}} + \operatorname{rk} Z(H_\ell(B)) \\
																				& = \operatorname{rk} H_\ell(A) + \operatorname{rk} H_\ell(B),
\end{aligned}
\]
and we conclude by lemma \ref{lemma_InclusionReductive}.
\end{proof}

\medskip

\section{Sufficient conditions for $H_\ell$ to decompose as a product}\label{sect_Products}

\subsection{An $\ell$-adic analogue of a theorem of Hazama}\label{sect_Hazamaladic}

We are now ready to prove the following $\ell$-adic analogue (and mild generalization) of a Hodge-theoretical result of Hazama (\cite[Proposition 1.8]{HazNonsimple}):

\begin{theorem}\label{thm_Hazama}
Let $K$ be a \ffn, $A_1$ and $A_2$ be $K$-abelian varieties, and $\ell$ be a prime number. For $i=1,2$ let $\mathfrak{h}_{i}$ be the Lie algebra of $H_\ell(A_i)$. Suppose that the following hold:
\begin{enumerate}
\item for $i=1,2$, the algebra $\mathfrak{h}_i$ is semisimple, so that we can write $\mathfrak{h}_{i} \otimes \mC \cong \mathfrak{h}_{i,1} \oplus \cdots \oplus \mathfrak{h}_{i,n_i}$, where every $\mathfrak{h}_{i,j}$ is simple;

\item for $i=1,2$, there exists a decomposition $V_\ell(A_i) \otimes \mC \cong V_{i,1} \oplus \cdots \oplus V_{i,n_i}$ such that the action of $\mathfrak{h}_{i} \otimes \mC \cong \mathfrak{h}_{i,1} \oplus \cdots \oplus \mathfrak{h}_{i,n_i}$ on $V_{i,1} \oplus \cdots \oplus V_{i,n_i}$ is componentwise and $\mathfrak{h}_{i,j}$ acts faithfully on $V_{i,j}$;

\item for all distinct pairs $(i,j)$ and $(i',j')$ for which there exists an isomorphism $\varphi: \mathfrak{h}_{i,j} \to \mathfrak{h}_{i',j'}$ there is an irreducible $\mathfrak{h}_{i,j}$-representation $W$ such that all simple $\mathfrak{h}_{i,j}$-submodules of $V_{i,j}$ and of $\varphi^*\left( V_{i',j'} \right)$ are isomorphic to $W$, and the highest weight defining $W$ is stable under all automorphisms of $\mathfrak{h}_{i,j}$.

%\item for every $(i,j)$, there exists a simple $\mathfrak{h}_{i,j}$-module $W_{i,j}$ such that all the simple $\mathfrak{h}_{i,j}$-submodules of $V_{i,j}$ are isomorphic to $W_{i,j}$;
%\item for every $(i,j)$, all the automorphisms of $\mathfrak{h}_{i,j}$ leave fixed the weight defining $W_{i,j}$;
%\item if $\mathfrak{h}_{i,j}$ and $\mathfrak{h}_{i',j'}$ are isomorphic, then so are $W_{i,j}$ and $W_{i',j'}$ (as $\mathfrak{h}_{i,j}$-representations).

\end{enumerate}

Then either $\operatorname{Hom}_{\barK}(A_1,A_2) \neq 0$ or $H_\ell(A_1 \times A_2) \cong H_\ell(A_1) \times H_\ell(A_2)$.
\end{theorem}

%\begin{remark}
%The reader will immediately notice that if $X_1,X_2$ are of prime dimension (or elliptic curves) and not of Type IV then the hypotheses are satisfied by the results of section \ref{sect_Results}.
%\end{remark}

\begin{remark}
Condition 3 is actually independent of the choice of the isomorphism $\varphi$: this follows easily from the fact that the highest weight of $W$ is stable under all automorphisms of $\mathfrak{h}_{i,j}$.
%Thanks to condition 4, condition 5 is independent of the isomorphism $\mathfrak{h}_{i,j} \cong \mathfrak{h}_{i',j'}$ chosen to compare the representations $W_{i,j}$ and $W_{i',j'}$.
\end{remark}

\begin{proof}
Let $\mathfrak{h}$ be the Lie algebra of $H_\ell(A_1 \times A_2)$. We shall try to apply lemma \ref{lemma_RibetLemma} to the inclusion $\mathfrak{h} \otimes {\mC} \hookrightarrow \left( \mathfrak{h}_{1} \oplus \mathfrak{h}_{2}\right) \otimes \mC$, and distinguish cases according to whether hypothesis (b2) is satisfied or not. Observe that $\mathfrak{h} \otimes \mC$ is a subalgebra of 
\[
\left(\mathfrak{h}_1 \oplus \mathfrak{h}_2\right) \otimes \mC \cong \bigoplus_{i=1}^2 \bigoplus_{j=1}^{n_i} \mathfrak{h}_{i,j} \subset \bigoplus_{i=1}^2 \bigoplus_{j=1}^{n_i} \mathfrak{gl}\left(V_{i,j}\right)
\]
whose projection on each factor $\mathfrak{gl}\left(V_{i,j}\right)$ is isomorphic to $\mathfrak{h}_{i,j}$, hence simple. Moreover, hypothesis 3 of this theorem implies condition (b1) of lemma \ref{lemma_RibetLemma}. Suppose now that (b2) holds as well: then $\mathfrak{h}\otimes {\mC} \cong \left(\mathfrak{h}_{1} \oplus \mathfrak{h}_{2} \right) \otimes \mC$, hence in particular $\operatorname{rk} \mathfrak{h} = \operatorname{rk} \mathfrak{h}_1 + \operatorname{rk} \mathfrak{h}_2$, and lemma \ref{lemma_InclusionReductive} implies $H_\ell(A_1 \times A_2) \cong H_\ell(A_1) \times H_\ell(A_2)$. Suppose on the other hand that (b2) fails: then there exists a nontrivial endomorphism $\varphi$ in
\[
\operatorname{End}_{\mathfrak{h} \otimes \mC}\left( \bigoplus_{i=1}^2 \bigoplus_{j=1}^{n_i} V_{i,j} \right) \setminus \bigoplus_{i=1}^2 \bigoplus_{j=1}^{n_i} \operatorname{End}_{\mathfrak{h}_{i,j}}\left( V_{i,j} \right).
\]

Since the action of $\mathfrak{h}_i \otimes \mC$ on $V_\ell(A_i) \otimes \mC \cong \bigoplus_{j=1}^{n_i} V_{i,j}$ is componentwise for $i=1,2$, it is clear that $\varphi$ does not belong to $\operatorname{End}_\mathfrak{h_1}\left( \bigoplus_{j=1}^{n_1} V_{1,j} \right) \times \left\{0\right\}$, nor to $\left\{0\right\} \times \operatorname{End}_\mathfrak{h_2}\left( \bigoplus_{j=1}^{n_2} V_{2,j} \right)$. Thus, up to exchanging the roles of $A_1$ and $A_2$ if necessary, the map $\varphi$ induces an $(\mathfrak{h} \otimes \mC)$-equivariant morphism from $\bigoplus_{j=1}^{n_1} V_{1,j}$ to $\bigoplus_{j=1}^{n_2} V_{2,j}$: this implies that the space
\[
\operatorname{Hom}_{\mathfrak{h}}\left(V_{\ell,1}, V_{\ell,2}\right) \otimes \mC \cong \operatorname{Hom}_{\mathfrak{h} \otimes \mC}\left(V_{\ell,1} \otimes \mC, V_{\ell,2} \otimes \mC \right)
\]
is nontrivial. In particular, $\operatorname{Hom}_{\mathfrak{h}}\left(V_{\ell}(A_1), V_{\ell}(A_2)\right) \neq 0$, and therefore $\operatorname{Hom}_{\barK}(A_1, A_2)$ is nontrivial by corollary \ref{cor_TateConjecture}. \end{proof}

\begin{remark}\label{rmk_Triality}
We now check to what extent the theorem can be applied to varieties $A$ that are of general Lefschetz type with respect to $\ell$ (definition \ref{def_GeneralLefschetz}). It is clear that conditions 1 and 2 are satisfied, so let us discuss condition 3. Let $\mathfrak{h}$ be a simple constituent of $\operatorname{Lie} H_\ell(A) \otimes \mC$. By definition, the simple $\mathfrak{h}$-submodules of $V_\ell(A) \otimes \mC$ are all isomorphic to a single representation $W$. Let us distinguish cases according to the type of $\mathfrak{h}$:
\begin{itemize}
\item if $\mathfrak{h}$ is of Lie type $A_l$, then $W$ is defined by the highest weight $\omega_{\frac{l+1}{2}}$ (recall that $l$ is odd by assumption), and is therefore stable under the unique nontrivial automorphism of the Dynkin diagram of $A_l$: condition 3 is satisfied;
\item if $\mathfrak{h}$ is of Lie type $B_l$ or $C_l$, the Dynkin diagram does not have any nontrivial automorphisms, hence all automorphisms of $\mathfrak{h}$ are inner and  fix the highest weight of $W$: condition 3 is again satisfied;
\item finally, if $\mathfrak{h}$ is of Lie type $D_l$ the module $W$ is defined by the highest weight $\omega_1$. As long as $l \neq 4$, the Dynkin diagram of $D_l$ has a unique nontrivial automorphism, and it is immediate to check that this automorphism fixes $\omega_1$: condition 3 is satisfied once more. Note however that for $l=4$ the Dynkin diagram has additional (triality) automorphisms, and that these do \textit{not} fix $\omega_1$, so condition 3 fails in this case.
\end{itemize}
%Finally, it is clear from the definition that conditions 1 through 3 are satisfied, and -- as long as 3 holds -- so is 5: w
Thus we conclude that every abelian variety $A$ of general Lefschetz type (at the prime $\ell$) satisfies the hypotheses of the previous theorem unless $\operatorname{Lie} H_\ell(A) \otimes \mC$ has a simple factor of Lie type $D_4$.
\end{remark}

\begin{corollary}\label{cor_ProductSurfaces}
Let $K$ be a \ff\ and $A_1, \ldots, A_n$ be absolutely simple abelian varieties defined over $K$, pairwise non-isogenous over $\barK$. Suppose that no $A_i$ is of type IV, and that the dimension of each $A_i$ is either 2 or an odd number. Let $k_1,\ldots,k_n$ be positive integers and $A$ be a $K$-abelian variety that is $\barK$-isogenous to $\prod_{i=1}^n A_i^{k_i}$. Then we have an isomorphism $H_\ell\left( A \right) \cong \prod_{i=1}^n H_\ell(A_i)$, and the Mumford-Tate conjecture holds for $A$.
\end{corollary}
\begin{proof} The Albert classification implies that every $A_i$ is of type I or II (recall that in characteristic zero there is no absolutely simple abelian surface of type III). As the three abelian varieties $\prod_{i=1}^n A_i^{k_i}$, $\prod_{i=1}^n A_i$ and $A$ all have the same Hodge group and the same groups $H_\ell$, there is no loss of generality in assuming that $k_1=\cdots=k_n=1$ and that $A=\prod_{i=1}^n A_i$. The equality $H_\ell(A_1 \times \cdots \times A_n) \cong H_\ell(A_1) \times \cdots \times H_\ell(A_n)$ then follows by induction from theorem \ref{thm_Hazama}, the hypotheses being verified thanks to theorem \ref{thm_BGK12} (and the remark following it). Lemma \ref{lemma_ProductImpliesMT} then implies that Mumford-Tate holds for $A_1 \times \cdots \times A_n$.
\end{proof}

\begin{corollary}\label{cor_OkWithECAndSurfaces}
Let $K$ be a \ff\ and $A_1, \ldots, A_n$ be absolutely simple $K$-abelian varieties of dimension at most 2, pairwise non-isogenous over $\barK$. Let $k_1,\ldots,k_n$ be positive integers and $A$ be a $K$-abelian variety that is $\barK$-isogenous to $\prod_{i=1}^n A_i^{k_i}$. Then we have $H_\ell\left( A \right) \cong \prod_{i=1}^n H_\ell(A_i)$, and the Mumford-Tate conjecture holds for $A$.
\end{corollary}

\begin{remark}\label{rmk_Shioda} 
Such a result is in a sense the best possible. There is an example -- due to Shioda \cite{Shioda81algebraiccycles} -- of an absolutely simple threefold $Y$ of CM type and a CM elliptic curve $E$ such that $H(Y \times E) \neq H(Y) \times H(E)$. By the Mumford-Tate conjecture in the CM case, this also means $H_\ell(Y \times E) \neq H_\ell(Y) \times H_\ell(E)$ (note that $Y$ and $E$, being CM, can be defined over a number field).
\end{remark}

\begin{proof}
As in the previous proof, we can assume $k_1=\cdots=k_n=1$ and replace $A$ by $\prod_{i=1}^n A_i$. By lemma \ref{lemma_ProductImpliesMT}, Mumford-Tate for $A$ would follow from the isomorphism $H_\ell\left( A \right) \cong \prod_{i=1}^n H_\ell(A_i)$, so let us prove the latter. Up to renumbering, we can also assume that $A_1, \ldots, A_m$ are of type I or II and $A_{m+1}, \ldots, A_n$ are of type IV (since there are no absolutely simple abelian varieties of type III of dimension at most 2). The classification of elliptic curves and simple surfaces implies that $A_{m+1}, \ldots, A_{n}$ are CM, because the endomorphism algebra of an absolutely simple abelian surface cannot be an imaginary quadratic field (\cite[§4]{shimura1963afp}). Let $A'=A_1 \times \cdots \times A_m$ and $A''=A_{m+1} \times \cdots \times A_n$. As $A''$ is CM and $A'$ has no simple factor of type IV, lemma \ref{lem_SemisimpleTimesCM} gives $H_\ell(A' \times A'') \cong H_\ell(A') \times H_\ell(A'')$. It thus suffices to prove the result when either $A'$ or $A''$ is trivial.

If $A''$ is trivial the claim follows from corollary \ref{cor_ProductSurfaces}, so we can assume $A'$ is trivial, in which case we have to show $H_\ell\left( \prod_{i=1}^n A_i \right) \cong \prod_{i=1}^n H_\ell(A_i)$ under the additional assumption that every $A_i$ is CM. Appealing to the Mumford-Tate conjecture in the CM case, it is enough to show the corresponding statement for Hodge groups, which is exactly the content of \cite[Theorem 3.15]{MR2384535}. 
\end{proof}

\subsection{A criterion in terms of relative dimensions}
As promised in the introduction, we have the following $\ell$-adic analogue of a theorem proved by Ichikawa in \cite{Ichikawa1991}:

\begin{theorem}\label{thm_Ichikawa_ell}
Let $K$ be a \ffn\ and $A'_i, A''_j$ (for $i=1,\ldots,n$ and $j=1,\ldots,m$) be absolutely simple $K$-abelian varieties of odd relative dimension that are pairwise non-isogenous over $\barK$. Suppose every $A'_i$ is of type I, II or III in the sense of Albert, and every $A''_j$ is of type IV. Let $A$ be a $K$-abelian variety that is $\overline{K}$-isogenous to $\prod_{i=1}^n A'_i \times \prod_{j=1}^m A''_j$: then
\[
H_\ell \left( A \right) \cong \prod_{i=1}^n H_\ell\left( A'_i\right) \times H_\ell\left( \prod_{j=1}^m A''_j \right).
\]
\end{theorem}

\medskip

For the proof of this theorem we shall need the following result:

\begin{proposition}\label{prop_AdmissibleAlgebras} Let $K$ be a \ffn, $A/K$ be an absolutely simple abelian variety of odd relative dimension and $\ell$ be a prime number. Write $\operatorname{Lie}(H_\ell(A)) \otimes \mC$ as $\mathfrak{c} \oplus \mathfrak{h}_1 \oplus \cdots \oplus \mathfrak{h}_n$, where $\mathfrak{c}$ is abelian and every $\mathfrak{h}_i$ is simple. Then
\begin{enumerate}
\item
if $A$ is of type I, II or III, then $A$ is of general Lefschetz type, and no simple factor $\mathfrak{h}_i$ is of Lie type $D_4$;%satisfies all the hypotheses of theorem \ref{thm_Hazama};
\item 
if $A$ is of type IV, then the algebras $\mathfrak{h}_i$ are of type $A_l$, where $l+1$ is not a power of $2$.
\end{enumerate}
\end{proposition}

%\begin{remark}\label{rmk_TypeIIIIII}
%In case 1, the abelian variety $A$ satisfies the hypotheses of theorem \ref{thm_Hazama} by remark \ref{rmk_Triality}.
%\end{remark}

\begin{proof}
Let $A$ be of type I, II or III. Then $A$ is of general Lefschetz type by theorem \ref{thm_BGK12} and proposition \ref{prop_TypeIII}, and again by proposition \ref{prop_TypeIII} the simple factors of $\operatorname{Lie}\left(H_{\ell}(A)\right) \otimes \mC$ of orthogonal type are of the form $\mathfrak{so}_{2h}$ with $h$ odd, so none of them is of Lie type $D_4$. %Hence $A$ satisfies the hypotheses of theorem \ref{thm_Hazama} by remark \ref{rmk_Triality}.

Let now $A$ be of type IV. Let $E$ be the center of the simple algebra $\operatorname{End}_{\barK}^0(A)$; set $e=[E:\mathbb{Q}]$ and $d^2=\left[\operatorname{End}_{\barK}^0(A):E \right]$. We are first going to show the desired property for those primes that split in $E$, and then extend the result to all primes through an interpolation argument based on the techniques of \cite{MR1074479}. Suppose therefore that $\ell$ is totally split in $E$. From the equality $E \otimes \mathbb{Q}_\ell \cong \mathbb{Q}_\ell^{[E:\mathbb{Q}]}$ we get
\[
\operatorname{End}_{\barK}^0(A) \otimes \mC \cong \bigoplus_{\sigma: E \hookrightarrow \mathbb{C}} M_d(\mC),
\]
so Schur's lemma implies
\[
V_\ell(A) \otimes \mC \cong \bigoplus_{\sigma: E \hookrightarrow \mathbb{C}}  W_\sigma^{\oplus d},
\]
where each $W_\sigma$ is simple of dimension $\frac{1}{de} \operatorname{dim}_{\mC}(V_\ell(A) \otimes \mC) = \operatorname{rel dim}(A)$. The action of $H_\ell(A)$ on $V_\ell(A)$ is faithful, so for every $i=1,\ldots,n$ there exists a $\sigma:E\hookrightarrow \mathbb{C}$ (depending on $i$) such that the action of $\mathfrak{h}_i$ is nontrivial on $W_\sigma$. Note that $\dim(W_\sigma)$ is odd. Let $W_\sigma \cong Z_1 \otimes \cdots \otimes Z_n$ be the decomposition of $W_\sigma$ with respect to the action of $\mathfrak{h}_1 \oplus \cdots \oplus \mathfrak{h}_n$; the module $Z_i$ is thus a nontrivial minuscule representation of $\mathfrak{h}_i$ of \textit{odd} dimension: since every minuscule module over an algebra of type $B_l, C_l, D_l$ is of \textit{even} dimension (cf. table 1), we deduce that $\mathfrak{h}_i$ is of type $A_l$ for a certain $l$. Furthermore, $l+1$ cannot be a power of 2, since in that case every irreducible minuscule module over $A_l$ is of even dimension. This shows our claim when $\ell$ is totally split.

\medskip
Let us now consider the general case. Let $\ell$ be any prime, and $p$ be a fixed prime that splits completely in $E$. Let $\Phi_\ell$ be the root system of $\left( G_\ell(A) \otimes \overline{\mathbb{Q}_\ell} \right)^{\operatorname{der}}$, and let $\Phi_\ell^0$ be the subset of $\Phi_\ell$ given by those roots that are short in their respective simple factors of $\left( G_\ell(A) \otimes \overline{\mathbb{Q}_\ell} \right)^{\operatorname{der}}$. Note that $\Phi^0_p=\Phi_p$, since $\Phi_p$ only involves root systems of type $A_l$ (and such root systems do not possess long roots). It is a theorem of Serre that the formal characters of the various $G_\ell(A)$, for varying $\ell$, are all equal (see \cite[Corollary 3.8]{Pink}), and from \cite[§4]{MR1074479} (see also pp. 212-213 of \cite{Pink}) we know that the formal character completely determines $\Phi^0_\ell$. Hence we have $\Phi^0_\ell = \Phi^0_p=\bigoplus_{i=1}^k A_{n_i}$ for a certain $k$ and for integers $n_i$ such that no $n_i+1$ is a power of 2; in particular, no $n_i$ equals 1. Write now $\Phi_\ell=\bigoplus_{i=1}^r R_i$, where each $R_i$ is a simple root system. It is easy to see that $A_l^0=A_l, B_l^0=lA_1, C_l^0=D_l$ and $D_l^0=D_l$, so the equality
\[
\bigoplus_{i=1}^k A_{n_i}=\Phi^0_p=\Phi^0_\ell=\bigoplus_{j=1}^r R_j^0
\]
implies -- by uniqueness of the decomposition in simple root systems -- that every root system $R_j$ is either of type $A_l$ or $B_m$ (for some $l,m$). On the other hand, if one $R_j$ were of type $B_m$, then the right hand side of the above equality would contain $B_m^0=m A_1$, but no root system of type $A_1$ can appear on the left hand side by what we have already shown. This implies that every $R_j$ is of type $A_l$ (for some $l$), and the uniqueness of the decomposition shows that $r=k$ and (up to renumbering the indices) $R_j=A_{n_j}$. Hence the root system of $G_{\ell}(A)^{\operatorname{der}}$ is the same as that of $G_{p}(A)^{\operatorname{der}}$, and in particular all the simple algebras $\mathfrak{h}_i$ are of Lie type $A_l$, where $l+1$ is not a power of 2.
\end{proof}

\begin{proof}{(of theorem \ref{thm_Ichikawa_ell})}
There is no loss of generality in assuming that $A=A' \times A''$, where
\[
A'=\prod_{i=1}^n A'_i, \quad A'' = \prod_{j=1}^m A''_j.%, \quad A=A' \times A''
\]
%Notice that $H_\ell(A')$ is isomorphic to $\prod_{i=1}^n H_\ell(A'_i)$: this follows by repeated application of theorem \ref{thm_Hazama}, whose hypotheses are satisfied thanks to remark \ref{rmk_Triality} and proposition \ref{prop_AdmissibleAlgebras} (indeed this proposition implies that every $A_i'$ is of general Lefschetz type and no algebra $\operatorname{Lie} \left(H_\ell(A'_i)\right) \otimes \overline{\mathbb{Q}_\ell}$ has a simple factor of Lie type $D_4$).
Repeatedly applying theorem \ref{thm_Hazama} shows that $H_\ell(A')$ is isomorphic to $\prod_{i=1}^n H_\ell(A'_i)$: indeed by proposition \ref{prop_AdmissibleAlgebras} we know that every $A_i'$ is of general Lefschetz type and no algebra $\operatorname{Lie} \left(H_\ell(A'_i)\right) \otimes \overline{\mathbb{Q}_\ell}$ has a simple factor of Lie type $D_4$, so the hypotheses of theorem \ref{thm_Hazama} are satisfied thanks to remark \ref{rmk_Triality}.
%By proposition \ref{prop_AdmissibleAlgebras} and remark \ref{rmk_Triality} we see that every $A'_i$ satisfies the hypotheses of theorem \ref{thm_Hazama}, so a simple induction shows that $H_\ell(A')$ is isomorphic to $\prod_{i=1}^n H_\ell(A'_i)$. 
%
%Notice first that theorem \ref{thm_Hazama} and an immediate induction imply that $H_\ell(A')$ is isomorphic to $\prod_{i=1}^n H_\ell(A'_i)$. 
Thus it is enough to show that $H_\ell(A) \cong H_\ell(A') \times H_\ell(A'')$, and this follows from proposition \ref{prop_DifferentSimpleFactors}: by the results of section \ref{sect_Results}, the simple factors of $\operatorname{Lie} \left( H_\ell(A') \right) \otimes \mC$ are either of type $\mathfrak{so}$, $\mathfrak{sp}$ or $\mathfrak{sl}_{l+1}$ (with $l+1$ a power of 2), whereas by the previous proposition the simple factors of $\operatorname{Lie} \left( H_\ell(A'')^{\operatorname{der}} \right) \otimes \mC$ are of type $\mathfrak{sl}_{l+1}$ (with $l+1$ not a power of 2).
\end{proof}

\begin{remark}
Notice that, as the rank of $H_\ell(A)$ is independent of $\ell$, knowing that part (2) of proposition \ref{prop_AdmissibleAlgebras} holds for \textit{some} prime $\ell$ would in fact be enough to prove theorem \ref{thm_Ichikawa_ell}. Though a weaker version of the proposition would be easier to show (since it would not require the second part of the proof provided), we have preferred to give and employ the result in its stronger form (applying to \textit{all} primes), which we believe has some merit in itself.
\end{remark}

\section{Results in positive characteristic}\label{sect_PosChar}
We now discuss the situation of $K$ being a field of positive characteristic, finitely generated over its prime field, and we restrict ourselves to the primes $\ell \neq \operatorname{char} K$. If $A$ is a $K$-abelian variety, we denote $G_\ell(A)$ the Zariski closure of the natural Galois representation
\[
\rho_\ell : \operatorname{Gal}\left(\kSep / K \right) \to \operatorname{Aut}\left( T_\ell(A) \right),
\]
where $\kSep$ is now a fixed \textit{separable} closure of $K$.

The main difficulty in translating the results of the previous sections to this context is that if we define $H_\ell(A)$ as $(G_\ell(A) \cap \operatorname{SL}(V_\ell(A)))^0$, then this group might not capture any information about $A$ at all. The crucial problem is the failure of Bogomolov's theorem in positive characteristic: for general abelian varieties $A/K$, it is not true that $G_\ell(A)$ contains the torus of homotheties, and therefore the intersection $G_\ell(A) \cap \operatorname{SL}(V_\ell(A))$ may very well be finite. 

\begin{remark}
A simple example of this phenomenon is given by an ordinary elliptic curve $E$ over a finite field $\mathbb{F}_q$. Let $\operatorname{Fr}_q$ be the Frobenius automorphism of $\mathbb{F}_q$; the image of $\rho_\ell$ is generated by the image $g$ of $\operatorname{Fr}_q$, and as it is well known we have $\det \rho_\ell(g) = q$. Looking at the Lie algebra of $G_\ell(E)$, it follows easily that this group is 1-dimensional and that $H_\ell(E)$ is the trivial group, so that no information about $E$ can be recovered from $H_\ell(E)$. This problem is studied in \cite{MR2289628}, where more examples of this situation are given.
\end{remark}

However, Zarhin has proved that a statement akin to Bogomolov's theorem holds in positive characteristic if we restrict ourselves to a certain (large) class of abelian varieties; more precisely, we have the following result:
\begin{theorem}{(\cite{MR0453757}, Theorem 2 and Corollary 1)}
Let $K$ be a finitely generated field of positive characteristic and $A$ be a $K$-abelian variety. Let $\ell$ be a prime different from $\operatorname{char}(K)$. There exist a semisimple Lie algebra $\mathfrak{h}$ and a 1-dimensional Lie algebra $\mathfrak{c}$ such that $\operatorname{Lie} G_\ell(A) \cong \mathfrak{c} \oplus \mathfrak{h}$.

If furthermore no simple factor of $A_{\overline{K}}$ is of type IV in the sense of Albert, then $\mathfrak{c} \cong \mathbb{Q}_\ell \cdot \operatorname{Id}$ is the Lie algebra of the torus of homotheties.
\end{theorem}

\begin{remark}
Zarhin's theorem is a rather direct consequence of the reductivity of $G_\ell(A)$ and of Tate's conjecture on homomorphisms. At the time of \cite{MR0453757}, these two facts had only been established (by Zarhin himself, cf. \cite{MR0371897} and \cite{MR0422287}) under the assumption that $\operatorname{char} K$ is greater than $2$, but Mori \cite{MR797982} has subsequently lifted this restriction.
\end{remark}

\begin{remark}\label{rmk_NotOkWith}
Let $K$ be a finitely generated field of positive characteristic and $E_1, E_2$ be two elliptic curves over $K$. Assume $\operatorname{End}_{\overline{K}}(E_1)$ and $\operatorname{End}_{\overline{K}}(E_1)$ are imaginary quadratic fields, and $E_1, E_2$ are not isogenous over $\overline{K}$. As $E_1 \times E_2$ is CM, the group $G_\ell(E_1 \times E_2)$ is abelian and therefore -- by Zarhin's theorem -- of dimension 1: this is in stark contrast with what happens in characteristic zero, where $H_\ell(E_1 \times E_2) \cong H_\ell(E_1) \times H_\ell(E_2)$ is of dimension 2. In particular, we cannot hope for an analogue of corollary \ref{cor_OkWithECAndSurfaces} to hold in positive characteristic.
\end{remark}

In view of Zarhin's theorem and of the previous remarks, the most natural definition for $H_\ell(A)$ in positive characteristic seems to be the following:
\begin{definition}\label{def_Hl_p}
Let $K$ be a finitely generated field of characteristic $p>0$. For every prime $\ell$ different from $p$ we set $H_\ell(A)=\left( G_\ell(A)^0 \right)^{\operatorname{der}}$.
\end{definition}

\begin{remark}
When the characteristic of $K$ is positive, Zarhin's theorem implies that $G_\ell(A)^{\operatorname{der}}$ is of codimension 1 in $G_\ell(A)$; this is not necessarily the case in characteristic zero. On the other hand, as in characteristic zero, it is clear from definition \ref{def_Hl_p} that $H_\ell(A \times B)$ projects surjectively onto $H_\ell(A)$ and $H_\ell(B)$.
\end{remark}

Let us now restrict ourselves to abelian varieties $A$ such that no simple factor of $A_{\overline{K}}$ is of type IV. In the proof of corollary \ref{cor_TateConjecture} we can then replace Bogomolov's theorem by Zarhin's theorem, at which point the argument used to show theorem \ref{thm_Hazama} goes through essentially unchanged. Thus for this class of abelian varieties we have:
\begin{theorem}{(cf. theorem \ref{thm_Hazama})}\label{thm_Hazama_p} Let $K$ be a finitely generated field of characteristic $p>0$ and $A_1, A_2$ be $K$-abelian varieties such that $A_{1,\overline{K}}$ and $A_{2,\overline{K}}$ have no simple factors of type IV. Let $\ell$ be a prime number different from $p$, and suppose hypotheses 1 through 3 of theorem \ref{thm_Hazama} are satisfied. Then either $\operatorname{Hom}_{\barK}(A_1,A_2) \neq 0$ or $H_\ell(A_1 \times A_2) \cong H_\ell(A_1) \times H_\ell(A_2)$.
\end{theorem}

\begin{remark}
This theorem is strictly weaker than the corresponding result in characteristic zero, in that there exist abelian varieties of type IV (over number fields) that satisfy all hypotheses of theorem \ref{thm_Hazama}. Examples of such varieties include fourfolds of type IV(1,1) that support exceptional Weil classes, cf. \cite{Moonen95hodgeand}. On the other hand, the abelian varieties of corollary \ref{cor_ProductSurfaces} satisfy the hypotheses of the present weakened version, hence the corollary remains true when $K$ is of positive characteristic.
\end{remark}

\smallskip

Let us now consider theorem \ref{thm_Ichikawa_ell}. Its proof essentially relies on theorem \ref{thm_BGK12} and proposition \ref{prop_TypeIII}, which in turn only depend on Tate's conjecture and on the minuscule weights conjecture (theorem \ref{thm_Pink}). As already remarked, the former is now known for arbitrary finitely generated fields of positive characteristic, while the second has been shown by Zarhin (\cite[Theorem 4.2]{MR740792}) under an additional technical assumption, namely that the abelian variety in question has ordinary reduction in dimension 1 at all places of $K$ with at most finitely many exceptions (cf.~\cite[Definition 4.1.0]{MR740792}; this is a condition weaker than being ordinary). Finally, for varieties of type IV we have also exploited the fact that the formal character of $G_\ell(A)^0$ is independent of $\ell$: this statement too is known for finitely generated fields of positive characteristic (see \cite{MR553707} and \cite{MR1150604}, Proposition 6.12 and Examples 6.2, 6.3), so proposition \ref{prop_AdmissibleAlgebras} is still valid in this context. Taking all these facts into account we obtain:

\begin{theorem}{(cf.~theorem \ref{thm_Ichikawa_ell})} \label{thm_Ichikawa_p}
Let $K$ be a finitely generated field of positive characteristic and $A'_i, A''_j$ (for $i=1,\ldots,n$ and $j=1,\ldots,m$) be absolutely simple $K$-abelian varieties of odd relative dimension that are pairwise non-isogenous over $\barK$. Suppose every $A'_i$ is of type I, II or III in the sense of Albert, and every $A''_j$ is of type IV. Finally, suppose that each $A'_i$ and each $A''_j$ has ordinary reduction in dimension 1 at all places of $K$ with at most finitely many exceptions, and let $\ell$ be a prime different from $\operatorname{char} K$. Let $A$ be a $K$-abelian variety that is $\overline{K}$-isogenous to $\prod_{i=1}^n A'_i \times \prod_{j=1}^m A''_j$: then
\[
H_\ell \left( A \right) \cong \prod_{i=1}^n H_\ell\left( A'_i\right) \times H_\ell\left( \prod_{j=1}^m A''_j \right).
\]
\end{theorem}

\section{Nonsimple varieties of dimension at most $5$}\label{sect_MoonenZarhin}
Let once more $K$ be a \ff\ and $A/K$ be an abelian variety. With the results of the previous sections at hand it is a simple matter to compute, when $A/K$ is of dimension at most 5 and nonsimple over $\overline{K}$, the structure of $H_\ell(A)$ in terms of the $H_\ell$'s of the simple factors of $A_{\barK}$. Given however that the analogous problem for $H(A)$ has been given a complete solution in \cite{MZ}, we limit ourselves to showing that (in most cases) an abelian variety $A$ of dimension at most 5 satisfies Mumford-Tate, and refer the reader to \cite{MZ} for more details on the precise structure of $H(A)$ (hence of $H_\ell(A)$). Note in any case that -- for many varieties, including those for which we cannot prove Mumford-Tate -- our arguments yield the structure of $H_\ell(A)$ directly, without appealing to the results of \cite{MZ}.

\begin{theorem}\label{thm_MT5}
Let $K$ be a \ff\ and $A$ be a $K$-abelian variety of dimension at most 5. Then the Mumford-Tate conjecture holds for $A$, except possibly in the following two cases:
\begin{enumerate}
\item $\dim A=4$ and $\operatorname{End}_{\overline{K}}(A)=\mathbb{Z}$;
\item $A$ is isogenous over $\overline{K}$ to a product $A_1 \times A_2$, where $A_1$ is an absolutely simple abelian fourfold with $\operatorname{End}_{\overline{K}}(A_1)=\mathbb{Z}$ and $A_2$ is an elliptic curve. In this case $H_\ell(A)$ is isomorphic to $H_\ell(A_1) \times H_\ell(A_2)$.
\end{enumerate}
%, $n$ be an integer no less than 2, and $A_1, \ldots, A_n$ be absolutely simple $K$-abelian varieties such that $\sum_{i=1}^n \dim A_i \leq 4$. Let $A$ be a $K$-abelian variety that is $\overline{K}$-isogenous to $A_1 \times \cdots \times A_n$: then the Mumford-Tate conjecture holds for $A$.
\end{theorem}
\begin{proof}
If $A$ is absolutely simple the result follows immediately from theorems \ref{thm_EC}, \ref{thm_PrimeDimension} and \ref{thm_MT4}. Suppose therefore that $A$ is not absolutely simple. Since $H(A)$ and $H_\ell(A)$ are invariant both under isogeny and finite extension of the base field, we can assume without loss of generality that $A$ is isomorphic to a product $A_1 \times \cdots \times A_n$, where each factor is absolutely simple. Furthermore, if all the $A_i$ are of dimension at most 2 we can simply apply corollary \ref{cor_OkWithECAndSurfaces}, so (up to renumbering) we can assume $\dim A_1 \geq 3$.

Consider first the case $\dim A=4$. By what we have already proved we can assume $A \cong A_1 \times A_2$, where $A_1$ is an absolutely simple threefold and $A_2$ is an elliptic curve. In particular, $A_1$ and $A_2$ are of odd relative dimension, so if $A_2$ does not have complex multiplication (hence it is not of type IV) we have $H_\ell(A_1 \times A_2) \cong H_\ell(A_1) \times H_\ell(A_2)$ by theorem \ref{thm_Ichikawa_ell}, and the claim follows from lemma \ref{lemma_ProductImpliesMT}. On the other hand, if $A_2$ does have complex multiplication the claim follows immediately from lemma \ref{lemma_CMImpliesMT}.

Next consider the case $\dim A=5$. 
We can assume that in the decomposition $A=A_1 \times \cdots \times A_n$ no two $A_i$'s are isogenous over $\overline{K}$, for otherwise the problem is reduced to a lower-dimensional one. Furthermore, we have already considered the case $n=1$, so we can also assume $n \geq 2$. Recall that we have renumbered the $A_i$ in such a way that $\dim A_1 \geq 3$.

Suppose first that at least one of the $A_i$ has complex multiplication. Write $A= B \times C$, where $C$ is the product of those $A_i$ that are CM and $B$ is the product of the remaining factors. We have $\dim B \leq 4$. If $B$ satisfies Mumford-Tate, then Mumford-Tate for $A$ follows from lemma \ref{lemma_CMImpliesMT} and we are done. If, on the contrary, $B$ does not satisfy Mumford-Tate, then the results of section \ref{sect_Results} together with the case $\dim A=4$ treated above imply that $B=A_1$ is an absolutely simple fourfold with $\operatorname{End}_{\barK}(B)=\mathbb{Z}$, and we are in case (2); hence we just need to prove that $H_\ell(A_1 \times A_2)$ is isomorphic to $H_\ell(A_1) \times H_\ell(A_2)$, which follows at once from lemma \ref{lem_SemisimpleTimesCM}. From now on we can therefore assume that no $A_i$ is CM. Also recall that elliptic curves and abelian surfaces without CM are of type I or II in the sense of Albert.

We now need to distinguish several sub-cases, each of which we shall treat by proving the equality $H_\ell(A) \cong \prod_{i=1}^n H_\ell(A_i)$: indeed, if Mumford-Tate holds for every $A_i$, this equality implies Mumford-Tate for $A$ by lemma \ref{lemma_ProductImpliesMT}, and if Mumford-Tate fails for one of the $A_i$'s this equality is all we have to show.

\smallskip

Suppose first that $\dim A_1=3$ and $A_2, A_3$ are elliptic curves (without CM): then for all primes $\ell$, and independently of the type of $A_1$, theorem \ref{thm_Ichikawa_ell} gives $H_\ell(A) \cong H_\ell(A_1) \times H_\ell(A_2) \times H_\ell(A_3)$.

\smallskip

Next suppose $\dim A_1$ is $3$ and $A_2$ is an absolutely simple abelian surface without CM (hence not of type IV). Let $\ell$ be any prime. If $\operatorname{rel dim}(A_2)=1$, or $A_1$ is not of type IV, then we have $H_\ell(A) \cong H_\ell(A_1) \times H_\ell(A_2)$ resp. by theorem \ref{thm_Ichikawa_ell} or corollary \ref{cor_ProductSurfaces}. We can therefore assume that $\operatorname{End}_{\barK}(A_2)$ is $\mathbb{Z}$ and $A_1$ is of type IV and does not have complex multiplication. It is known that in this case $\operatorname{Lie}(H_\ell(A_2)) \cong \mathfrak{sp}_{4,\mathbb{Q}_\ell}$, and $\operatorname{Lie} \left(H_\ell(A_1)^{\operatorname{der}}\right) \otimes \mC \cong \mathfrak{sl}_{3,\mC}$ (cf. \cite{Ribet83classeson}), so it follows from proposition \ref{prop_DifferentSimpleFactors} that $H_\ell(A) \cong H_\ell(A_1) \times H_\ell(A_2)$.

\smallskip

We now need to consider the case when $A_1$ is an absolutely simple abelian fourfold and $A_2$ is an elliptic curve without CM; this assumption will be in force for the remainder of the proof.

Suppose first that $A_1$ is not of type IV and that $\operatorname{End}_{\barK}(A_1) \neq \mathbb{Z}$. By the results of \cite{Moonen95hodgeand} we know that $A_1$ is of general Lefschetz type, so that the equality $H_\ell(A_1 \times A_2) \cong H_\ell(A_1) \times H_\ell(A_2)$ follows from theorem \ref{thm_Hazama} and remark \ref{rmk_Triality}.

\smallskip

Consider now the case when $A_1$ is of type IV. It is not hard to check (from the results in \cite{Moonen95hodgeand}) that either $\operatorname{Lie}(H_\ell(A_1)) \otimes \mC$ does not have any simple factor isomorphic to $\mathfrak{sl}_2$ (cases IV(1,1) and IV(4,1) in the notation of \cite{Moonen95hodgeand}) or we are in case IV(2,1). In the former case we apply proposition \ref{prop_DifferentSimpleFactors} to deduce that $H_\ell(A) \cong H_\ell(A_1) \times H_\ell(A_2)$ for all primes $\ell$. Suppose instead that we are in case IV(2,1), that is to say $\operatorname{End}^0_{\barK}(A_1)$ is a CM field $E$ of degree 4 over $\mathbb{Q}$. Let $E_0$ be the maximal totally real subfield of $E$. We read from \cite{Moonen95hodgeand} the equality $H(A_1)^{\operatorname{der}}=\operatorname{Res}_{E_0/\mathbb{Q}} \operatorname{SU}(E^2,\psi)$, where $\psi$ is a suitable Hermitian form on $E^2$. Since $[E_0:\mathbb{Q}]=2$ and $\operatorname{SU}(E^2,\psi)$ is an $E_0$-form of $\operatorname{SL}_2$, the group $H(A_1)^{\operatorname{der}}$ is isogenous to a $\mathbb{Q}$-form of $\operatorname{SL}_2^2$; moreover, it is $\mathbb{Q}$-simple by Theorem 1.10 of \cite{Pink2004}. Finally, the Mumford-Tate conjecture holds for $A_1$ by theorem \ref{thm_MT4}, so for all primes $\ell$ we have an isomorphism $H_\ell(A_1) \cong H(A_1) \otimes \mathbb{Q}_\ell$. By lemma \ref{lemma_CanChooselSoGroupIsSimple} there is a prime $p$ such that the group $H_p(A_1)^{\operatorname{der}} \cong H(A_1)^{\operatorname{der}} \otimes \mathbb{Q}_p$ is simple over $\mathbb{Q}_p$.
Suppose by contradiction $\operatorname{rk} H_p(A)^{\operatorname{der}} < \operatorname{rk} H_p(A_1)^{\operatorname{der}} + \operatorname{rk} H_p(A_2)$. As $\operatorname{rk} H_p(A_2)=1$ we have $\operatorname{rk} H_p(A)^{\operatorname{der}} = \operatorname{rk} H_p(A_1)^{\operatorname{der}}$, so the natural projection $H_p(A) \twoheadrightarrow H_p(A_1)$ induces an isogeny $H_p(A)^{\operatorname{der}} \to H_p(A_1)^{\operatorname{der}}$. Composing the dual isogeny $H_p(A_1)^{\operatorname{der}} \to H_p(A)^{\operatorname{der}}$ with the projection of $H_p(A)^{\operatorname{der}}$ onto $H_p(A_2)$ we obtain a surjective morphism $H_p(A_1)^{\operatorname{der}} \twoheadrightarrow H_p(A_2)$: but this is absurd, because the two groups have different ranks and $H_p(A_1)^{\operatorname{der}}$ is simple. The contradiction shows that $\operatorname{rk} H_p(A)^{\operatorname{der}} = \operatorname{rk} H_p(A_1)^{\operatorname{der}} + \operatorname{rk} H_p(A_2)$, from which we deduce first that $H_p(A) \cong H_p(A_1) \times H_p(A_2)$ and then (since the ranks of $H_\ell(A_1), H_\ell(A_2)$ and $H_\ell(A)$ do not depend on $\ell$) that $H_\ell(A) \cong H_\ell(A_1) \times H_\ell(A_2)$ holds for all primes $\ell$.

\smallskip

We finally come to the case $\dim A_1=4$ and $\operatorname{End}_{\barK}(A_1) = \mathbb{Z}$. If for one (hence every) prime $\ell$ we have $H_\ell(A_1)=\operatorname{Sp}_{8,\mathbb{Q}_\ell}$, then the abelian variety $A_1$ is of general Lefschetz type (cf. \cite[§4.1]{Moonen95hodgeand}), so the equality $H_\ell(A) \cong H_\ell(A_1) \times H_\ell(A_2)$ follows from theorem \ref{thm_Hazama}. Thus the last case we have to cover is that of $H_\ell(A_1)$ being isogenous to a $\mathbb{Q}_\ell$-form of $\operatorname{SL}_2^3$ for every prime $\ell$. By \cite[Theorem 5.13]{Pink}, there is a simple $\mathbb{Q}$-algebraic group $P(A_1)$ such that, for a set of primes $\ell$ of Dirichlet density 1, we have an isomorphism $H_\ell(A_1) \cong P(A_1) \otimes \mathbb{Q}_\ell$. Furthermore, $P(A_1)$ is isogenous to a $\mathbb{Q}$-form of $\operatorname{SL}_2^3$, so by lemma \ref{lemma_CanChooselSoGroupIsSimple} we can choose a prime $p$ for which $H_p(A_1) \cong P(A_1) \otimes \mathbb{Q}_p$ is $\mathbb{Q}_p$-simple. We can now repeat the argument of case IV(2,1) above: if by contradiction we had $\operatorname{rk} H_p(A) < \operatorname{rk} H_p(A_1) + \operatorname{rk} H_p(A_2)$ we would obtain a surjective morphism from $H_p(A_1)$ to $H_p(A_2)$, which is absurd since the two groups have different rank and the source is simple. We deduce once more that $\operatorname{rk} H_\ell(A) = \operatorname{rk} H_\ell(A_1) + \operatorname{rk} H_\ell(A_2)$ holds for $\ell=p$ (hence for every prime $\ell$), so for every $\ell$ we have $H_\ell(A) \cong H_\ell(A_1) \times H_\ell(A_2)$.
\end{proof}

\bibliography{Biblio}{}
\bibliographystyle{plain}

\end{document}